\documentclass{amsart}
\usepackage[usenames, dvipsnames]{color}
\definecolor{purp}{rgb}{0.73, 0.26, 0.7}

\usepackage{amssymb,amsmath, amsthm, amsfonts,color}
\usepackage{mathrsfs, stmaryrd}
\usepackage[all]{xypic}
\usepackage{hyperref}
\usepackage{tikz}

\newtheorem{theorem}{Theorem}[section]
\newtheorem{lemma}[theorem]{Lemma}
\newtheorem{corollary}[theorem]{Corollary}
\newtheorem{proposition}[theorem]{Proposition}
\newtheorem{fact}{Fact}

\newtheorem{definition}[theorem]{Definition}
\newtheorem{remark}[theorem]{Remark}
\newtheorem{example}[theorem]{Example}
\newtheorem{notation}{Notation}

\newcommand{\luk}{\L u\-ka\-si\-e\-w\-icz}
\newcommand{\FP}{\mathcal{F}_{\mathbb{P}}}

\newcommand{\rest}{\upharpoonright}
\newcommand{\supp}{{\rm supp}}
\newcommand{\FPL}{FP(\Pi, \textrm{\L}{_\Delta})}

\begin{document}


\title[States of free product algebras]{{\small Towards a probability theory for product logic: states, integral representation and reasoning}}

\author{Tommaso Flaminio}
\address{Department of Pure and Applied Sciences, University of Insubria. Via Mazzini 5, 21100 Varese, Italy. Email: {\tt tommaso.flaminio@uninsubria.it}}
\author{Lluis Godo}
\address{Artificial Intelligence Research Institute (IIIA - CSIC), Campus de la Univ. Autonoma de Barcelona s/n, 08193 Bellaterra, Spain. Email: {\tt godo@iiia.csic.es}}
\author{Sara Ugolini}
\address{Department of Computer Science, University of Pisa. Largo B. Pontecorvo, 3
56127 Pisa, Italy. Email: {\tt sara.ugolini@di.unipi.it}}

\begin{abstract}
The aim of this paper is to extend probability theory from the classical to the product t-norm fuzzy logic setting. More precisely, we axiomatize a generalized notion of finitely additive probability for 
product logic formulas,  called state,  and 
show that every state is the Lebesgue integral with respect to
a unique regular Borel  
probability measure.  
Furthermore, the relation between states and measures is shown to be one-one. 
In addition, we study geometrical properties of the convex set of states 
and show that extremal states, i.e., the extremal points of the state space, are the same as the truth-value assignments of the logic.
Finally,  we  axiomatize 
a two-tiered modal logic for probabilistic reasoning on product logic events and prove soundness and completeness with respect to probabilistic spaces, where the algebra is a free product algebra and the measure is a state in the above sense. 
\vspace{.2cm}

\end{abstract}





\maketitle
\section{Introduction}
In his monograph \cite{H98}, H\'ajek established the theoretical ground for a wide family of fuzzy (thus, many-valued) logics which, since then, has been significantly developed and further generalized,  
giving rise to a discipline that has been named Mathematical Fuzzy logic (MFL).  H\'ajek's approach consists in fixing the real unit interval as  standard domain to evaluate atomic formulas, while the evaluation of compound sentences only depends on the chosen operation which provides the semantics for the so called {\em strong conjunction} connective. His general approach to fuzzy logics is grounded on the observation that, if strong conjunction is interpreted by a continuous t-norm \cite{KMP},  then any other connective of a logic has a natural standard interpretation.

Among continuous t-norms, the so called  {\L}ukasiewicz, G\"odel and product t-norms play a fundamental role. Indeed, Mostert-Shields' theorem \cite{KMP} shows that a  t-norm is continuous if and only if it can be built from the previous three ones by the construction of ordinal sum. In other words, a t-norm is continuous if and only if it is an ordinal sum of  \L ukasiewicz, G\"odel and product t-norms. These three operations determine three different algebraizable propositional logics (bringing the same names as their associated t-norms), whose equivalent algebraic semantics are the varieties of MV, G\"odel and product algebras respectively.    

The first  generalization of probability theory to the nonclassical settings of t-norm based fuzzy logics in H\'ajek sense, is due to Mundici who, in 1995, introduced the notion of {\em state} for the class of MV-algebras ---the algebraic counterpart  \luk\ logic--- with the aim of capturing the notion of average degree of truth of a proposition,  \cite{Mu}.

In that paper, states are functions mapping an MV-algebra to the real unit interval $[0, 1]$, satisfying a normalization condition and the finite additivity law. Such functions suitably generalize the classical notion of finitely additive probability measures on Boolean algebras, in addition to corresponding to convex combinations of valuations of \luk\ propositional logic. However, states and probability measures were previously studied in \cite{BuKl} (see also \cite{BuKl2,Navara}) on {\L}ukasiewicz tribes ($\sigma$-complete MV-algebras of fuzzy sets) as well as on other t-norm based tribes with continuous operations. 

MV-algebraic states have been deeply studied in recent years, as they enjoy several important properties and characterizations (see \cite{FK15} for a survey).
One of the most important results  in that framework 
is Kroupa-Panti theorem \cite[\S10]{Mu12}, a representation result  showing that every state of an MV-algebra is the Lebesgue integral with respect to a regular Borel probability measure. 
Moreover, the correspondence between states and regular Borel probability measures is one-one.

Many attempts of defining suitable notions of state in different structures have been made (see for instance \cite[\S8]{FK15} for a short survey). In particular, in \cite{AGM1}, the authors 
provide a definition of state for the Lindenbaum algebra of G\"odel logic that corresponds to the integration of the $n$-place truth-functions corresponding to G\"odel formulas, with respect to Borel probability measures on the real unit cube $[0,1]^{n}$. Moreover, such states are shown to correspond to convex combinations of finitely many truth-value assignments. Similar results have been obtained for the case G\"odel logic expanded with Baaz-Monteiro operator $\Delta$ \cite{ABGV17}, and for the case of Nilpotent Minimum logic \cite{AG10}. 

The aims of this contribution are the following: (1) we will introduce and study 
 states for product logic --the remaining fundamental many-valued logic for which such a notion is still lacking-- (2) we will prove that our axiomatization results in characterizing Lebesgue integrals of truth-functions of product logic formulas  with respect to regular Borel probability measures, and (3) following similar lines to those of \cite{FG07,HGE}, we will axiomatize a modal expansion of \luk\ logic for probabilistic reasoning on events  described by formulas of product logic. 
  In more detail, we show that states of the Lindenbaum algebra of product logic over $n$ variables, i.e. the free $n$-generated product algebra, correspond, one-one, to regular Borel probability measures on $[0,1]^n$.\footnote{Note that, unlike Kroupa-Panti theorem, we do not deal with states of {\em any} product algebra but of finitely-generated free product algebras.}
Moreover, and quite surprisingly since in the axiomatization of states the product t-norm operation is only indirectly involved via a condition concerning double negation, 
we prove that every state belongs to the convex closure of product logic valuations. Finally, these results will allow us to introduce a suitable class of probabilistic-like models with respect to which the modal logic we will introduce in Section \ref{sec:logic} turns out to be sound and complete.

The paper is structured as follows. After this introduction, in Section \ref{sec:pre} we will recall the functional representation of the free $n$-generated product algebra $\mathcal{F}_{\mathbb{P}}(n)$, as presented in \cite{ABG} (see also \cite{CG}). We will easily prove that such functions, although they are not continuous, are indeed Borel measurable.  In particular,  from that functional representation of product logic functions, it follows that the domain $[0,1]^n$ of each such a function can be partitioned in
 locally compact and Hausdorff subsets of $[0,1]^{n}$, named $G_{\varepsilon}$ (with $\varepsilon$ varying in a certain set $\Sigma$, depending on the atoms of the boolean skeleton of $\mathcal{F}_{\mathbb{P}}(n)$). More precisely, each $G_{\varepsilon}$ is an $F_\sigma$ set since, in fact, it is a countable union of a family $\{G_{\varepsilon}^{q}\}_{q \in (0, 1] \cap \mathbb{Q}}$ of nested compact subsets of $[0,1]^{n}$, and hence it is a $\sigma$-locally compact set (see \cite[\S 1.11]{Ru}). Over each $G_{\varepsilon}$, the function is actually continuous. Moreover, any continuous function with domain one of the compact sets $G_{\varepsilon}^{q}$ can be uniformly approximated by linear combinations of the functions of $\mathcal{F}_{\mathbb{P}}(n)$ restricted to such subsets.

In Section \ref{sec:states}, we will axiomatize our notion of state of $\mathcal{F}_{\mathbb{P}}(n)$, and show its properties together with some examples. In particular, we will investigate states of the 1-generated free product algebra, and see how this analysis reflects into its spectral space. 

In Section \ref{sec:integral} we will prove our main result, that is to say, for every state $s$ of $\mathcal{F}_{\mathbb{P}}(n)$ there is a unique Borel probability measure $\mu$ on $[0,1]^{n}$ such that $s$ is the Lebesgue integral with respect to $\mu$, and viceversa, every such an integral operator is a state in our sense.
In Section \ref{sec:extremal}, we shall prove that 
the state space of $\mathcal{F}_{\mathbb{P}}(n)$ is  convex and closed. Thus, via Krein-Milman theorem (see for instance \cite{Goodearl}) every state is a convex combination of extremal ones. We will hence characterize the extremal states, proving that they coincide with the homomorphisms of $\FP(n)$ into $[0,1]$, that is to say,  product logic valuations. Thus, the state space results to be generated by the truth-value assignments of the logic.

Finally, Section \ref{sec:logic} is devoted to presenting a logic for probabilistic reasoning on many-valued events represented by formulas of product logic. For that formalism we will provide an (infinitary) axiomatization which is sound and complete with respect to a probabilistic-like semantics given by states of free product algebras.

For the sake of readability we moved some technical proofs in an appendix at the end of the paper.

\section{Product algebras and product functions}\label{sec:pre}
 In this section we are going to recall  some basic facts and preliminary notions about  product algebras. In particular we will focus on  free, finitely generated, product algebras and their functional representation, mainly reporting results from \cite{ABG}. We assume the reader to be familiar with standard notions of universal algebra and algebraic semantics for many-valued logics. For otherwise we point them to the standard monographs \cite{BS} and \cite{H98,Hand1} respectively. To start with, let us recall that a BL-algebra \cite{H98} is a bounded, integral and commutative residuated lattice  ${\bf A}=(A,\odot, \to, \wedge,\vee,0,1)$ which  satisfies the following equations:
\begin{itemize}
\item[] $(x\to y)\vee (y\to x)=1$ (prelinearity),
\item[] $x\odot(x\to y)=x\wedge y$ (divisibility).
\end{itemize}
In what follows we shall adopt the following abbreviations: $\neg x:=x\to 0$, for every $n\in \mathbb{N}$, $x^n:= x\odot\ldots\odot x$ ($n$-times).

A BL-algebra ${\bf A}$ is a {\em product algebra} if it further satisfies
\begin{center}
 $x\wedge \neg x=0$ and 
 $\neg\neg x\to((y\odot x\to z\odot x)\to (y\to z))=1$.
\end{center}
Product algebras form a variety which is denoted by $\mathbb{P}$. 

\begin{example}\label{ex:productAlgebras}
(1) Any Boolean algebra is a product algebra. Furthermore for every product algebra ${\bf A}$, the biggest Boolean subalgebra of ${\bf A}$, $\mathscr{B}({\bf A})$ has universe $\{x\in A\mid \neg\neg x=x\}$ (cf. \cite[Theorem 3.1(1)]{MU15}). The Boolean algebra $\mathscr{B}({\bf A})$ is called the {\em Boolean skeleton} of ${\bf A}$.
\vspace{.2cm}

\noindent(2) Endow the real unit interval $[0,1]$ with operations defined in the following manner: $x\odot y=x\cdot y$ (the usual product), $x\to y=1$ if $x\leq y$ and $x\to y=y/x$ otherwise, $x\wedge y=\min(x,y)$, $x\vee y=\max(x,y)$. Thus $[0,1]_\Pi =([0,1], \odot, \to, \wedge, \vee, 0,1)$ is a product algebra, known as the {\em standard} product algebra. Any non Boolean product algebra, such as $[0,1]_\Pi$, is generic for $\mathbb{P}$, i.e., $\mathbb{P}$ is generated as a variety by $[0,1]_\Pi$ \cite[Corollary 4.1.11]{H98}.

%
\end{example}
For every $n\in\mathbb{N}$, let $\FP(n)$ be the free product algebra over $n$ free generators. That is, since $[0,1]_\Pi$ is generic for $\mathbb{P}$, $\FP(n)$ is  isomorphic to the product subalgebra of $[0,1]^{[0,1]^n}$ generated by the projection maps (see \cite{ABG}).  Thus, every element of $\FP(n)$ can be regarded as a function $[0,1]^n\to[0,1]$ that  we shall call a {\em product function}.

We are going to recall the description of $\FP(n)$ as presented in \cite{ABG} of which we will also adopt the notation for the sake of uniformity. It is known 
that $\mathscr{B}(\FP(n))$, the Boolean skeleton of $\FP(n)$, coincides with the free Boolean algebra over $n$ generators. In particular, $\mathscr{B}(\FP(n))$ is finite and hence atomic. Thus, we can safely identify the set of atoms of $\mathscr{B}(\FP(n))$, say $at(n)$, with the set $\Sigma=\{1,2\}^n$ of strings $\epsilon = (\epsilon_{1}, \ldots, \epsilon_{n})$ of length $n$ over the binary set $\{1,2\}$ and adopt the same notation of \cite{ABG} without danger of confusion:
$$
at(n)=\{p_\epsilon\in \mathcal{F}_{\mathbb{B}}(n)\mid \epsilon\in \Sigma\},
$$ 
that is to say, for every $\epsilon=\langle\epsilon_1,\ldots, \epsilon_n\rangle\in \Sigma$, $p_\epsilon=\bigwedge_{i=1}^n \neg^{\epsilon_i}x_i$, where $x_1, \ldots, x_n$ are the free generators of $\FP(n)$, while $\neg^1=\neg$ and $\neg^2=\neg\neg$.
Thus, for every $\epsilon=\langle\epsilon_1,\ldots, \epsilon_n\rangle$, we define
$$
G_\epsilon=\{\langle t_1,\ldots, t_n\rangle\in [0,1]^n\mid t_i>0 \mbox{ if }\epsilon_i=2\mbox{ and }t_i=0\mbox{ if }\epsilon_i=1\}.
$$ 
The set $\{G_\epsilon\mid \epsilon\in \Sigma\}$  is then a
partition of $[0,1]^n$ (cf. \cite{ABG}).  

For instance, for $n=2$, we have $p_{(1,1)} = \neg x_1 \land \neg x_2,\, p_{(1,2)} = \neg x_1 \land \neg \neg x_2,\, p_{(2,1)} = \neg \neg x_1 \land \neg x_2,\, p_{(2,2)} = \neg \neg x_1 \land \neg \neg x_2$,  
while the following Figure \ref{figG} shows how $[0,1]^2$ is partitioned by $G_{(1,1)}$, $G_{(1,2)}$, $G_{(2,1)}$ and $G_{(2,2)}$. 

\begin{figure}[h]
\begin{center}
\includegraphics[scale=0.5]{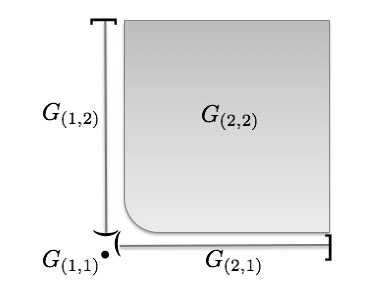}
\caption{The partition of $[0,1]^2$ into $G_{(1,1)}$, $G_{(1,2)}$, $G_{(2,1)}$ and $G_{(2,2)}$.}\label{figG}
\end{center}
\end{figure}

In what follows, for every function $f:[0,1]^n\to[0,1]$ and for every $\epsilon\in \Sigma$, we will denote by $f_\epsilon$  the restriction of $f$ to $G_\epsilon$, i.e. $f_\epsilon = f_{\rest G_\epsilon}$.
\begin{definition}
Let $n\in \mathbb{N}$ and $\mathscr{P}(n)$ be the set of functions $f:[0,1]^n\to[0,1]$ such that, for every $ \epsilon\in \Sigma$,  either $f_\epsilon=0$  or, if $f_\epsilon>0$ pointwise, it is continuous and piecewise monomial.
The pointwise application of the operations $\odot$, $\to$, $\wedge$ and $\vee$, together with the functions constantly $0$ and $1$ make $\mathscr{P}(n)$ into a product algebra that we still  denote by $\mathscr{P}(n)$ without danger of confusion.
\end{definition}

The functional representation theorem for free, finitely generated, product algebras then reads as follows:
\begin{theorem}[\cite{ABG}]\label{thm:repProd}
For every $n\inÊ\mathbb{N}$, $\FP(n)$ is isomorphic to $\mathscr{P}(n)$. Thus a function $f:[0,1]^n\to[0,1]$ is a product function iff $f$ is such that, for every $\epsilon\in \Sigma$, either $f_\epsilon=0$ or $f_\epsilon>0$ and it is continuous and  piecewise monomial. 
\end{theorem}
%
In the rest of this section we shall provide preparatory results about the
sets $G_\epsilon$. 
\begin{lemma}\label{lemma:borel}
For every $\epsilon \in \Sigma$, $G_\epsilon$ is a Borel subset of $[0,1]^n$, locally compact and Hausdorff.
\end{lemma}

\begin{proof} 
First of all, by definition,  $$G_\epsilon= \prod_{i=1}^n A_i, $$
where $A_i =(0, 1]$ if $\epsilon_i=2$ and $A_i = \{0\}$ if $\epsilon_i=1$. Hence, $G_\epsilon$ can also be expressed as the following countable union of closed sets in the product topology of $[0, 1]^n$, 
\begin{equation}\label{eqGe}
G_\epsilon = \bigcup_{q \in \mathbb{Q} \cap (0, 1]}  \prod_{i=1}^n B^q_i, 
\end{equation}
where $B^q_i = [q, 1]$ if $\epsilon_i=2$ and $B^q_i = \{0\}$ if $\epsilon_i=1$. Therefore  $G_\epsilon$ is a Borel subset of $[0,1]^n$. 
 It also easily follows that each $G_{\epsilon}$ is locally compact and Hausdorff.
\end{proof}
In the  proof of Lemma \ref{lemma:borel}  above, we showed that every $G_\epsilon$ is a countable union of compact subsets of $[0,1]^n$, through (\ref{eqGe}).
For the sake of a later use and a lighter notation, let us introduce the following.
\begin{notation}
For every $\epsilon\in \Sigma$ and for every $q\in (0,1]\cap\mathbb{Q}$, 
$$
G_\epsilon^q=\prod_{i=1}^n B_i^q.
$$ 
\end{notation}
\begin{remark}\label{rem:nested}
For every $\epsilon\in\Sigma$, the set $\{G_\epsilon^q\mid q\in (0,1]\cap \mathbb{Q}\}$ is a countable family of compact subsets of $G_\epsilon$ and 
$$
G_\epsilon=\bigcup_{q\in (0,1]\cap\mathbb{Q}} G_\epsilon^q .
$$ 
Therefore, each $G_\epsilon$ is $\sigma$-locally compact and Hausdorff (see \cite{Willard} for further details about $\sigma$-compact spaces).
Moreover, for each $q_1, q_2\in(0,1]\cap \mathbb{Q}$, $G_\epsilon^{q_1}\subset G_\epsilon^{q_2}$ iff $q_1>q_2$.  
\end{remark}

Given the previous result, we can easily prove what follows.
\begin{theorem}
Every $f \in \mathscr{P}(n)$ is measurable. 
\end{theorem}
\begin{proof}
We can write each $f \in \mathscr{P}(n)$ as $ f = \bigvee_{\epsilon} (f \land p_{\epsilon})$, where the restriction of each $f \land p_{\epsilon}$ to $G_\epsilon$ is either $0$ or is a piecewise monomial function. By Lemma \ref{lemma:borel}, each $G_{\epsilon}$ is a Borel set. Thus, each $f \land p_{\epsilon}$ is continuous on a Borel set, and $0$ outside, hence it is Borel measurable. The supremum of measurable functions is measurable, thus the claim follows.
\end{proof}

Now, we need to introduce some more notation. For every $n\in \mathbb{N}$ and for every $\epsilon\in \Sigma$, let:
\begin{itemize}
\item[(1)] $\mathscr{P}_\epsilon(n)=\{f_\epsilon:G_\epsilon\to[0,1]\mid f\in \mathscr{P}(n)\}$,
\item[(2)] $\mathscr{L}_\epsilon(n)=\{g:G_\epsilon\to \mathbb{R}\mid (\exists\lambda_1,\ldots,\lambda_k\in \mathbb{R}\setminus\{0\}) \& 
(\exists f_{1,\epsilon}, \ldots, f_{k, \epsilon}\in \mathscr{P}_\epsilon(n)) \mbox{ such that } g=\sum_{i=1}^k\lambda_i f_{i,Ê\epsilon}\}$.
\end{itemize}
In other words, $\mathscr{P}_\epsilon(n)$ is obtained by restricting each function of $\mathscr{P}(n)$ to $G_\epsilon$, while $\mathscr{L}_\epsilon(n)$ is in fact the {\em linear span} of $\mathscr{P}_\epsilon(n)$ with nonzero coefficients. 

\begin{proposition}\label{prop:unique}
For every  $\epsilon \in \Sigma$ and every $g_\epsilon\in \mathscr{L}_\epsilon(n)$, either $g_\epsilon=0$ or $g_\epsilon$ is a  piecewise polynomial function. Moreover, in the latter case, 
$g_\epsilon$ is represented in a unique way as a linear combination of pairwise distinct $f_{1,\epsilon},\ldots,f_{k,\epsilon}\in \mathscr{P}_\epsilon(n)$ with non-zero coefficients.
\end{proposition}
\begin{proof}
Assume $g_\epsilon\neq 0$. Then there is a $k\in \mathbb{N}$, $f_{1,\epsilon},Ê\ldots, f_{k,\epsilon}\in \mathscr{P}_\epsilon(n)$ and $\lambda_1,\ldots,\lambda_k\in \mathbb{R}\setminus\{0\}$ such that  $g_\epsilon=\sum_i \lambda_i\cdot f_{i,\epsilon}$. Each $f_{i,\epsilon}$ is piecewise monomial, meaning that there is a 
partition $\mathcal{P}_i=\{P_{1,i}, \ldots, P_{m_i, i}\}$ of $G_\epsilon$  such that the restricted function $f_{i,\epsilon}{\rest_{P_{j,i}}}$ is monomial. Let $\{Q_1,\ldots, Q_m\}$ be the refined partition of $G_\epsilon$ obtained by taking all possible non-void intersections of elements in the $\mathcal{P}_i$'s. Obviously, $f_{i,\epsilon}{\rest_{Q_j}}$ is monomial for all $Q_j$. 
Moreover, 
$$
g_\epsilon{\rest_{Q_j}}=\left(\sum_i \lambda_i f_{i,\epsilon}\right){\rest_{Q_j}}=\sum_i \lambda_i (f_{i,\epsilon}{\rest_{Q_j}}),
$$
whence $g_\epsilon{\rest_{Q_j}}$ is polynomial. Thus, the claim follows since a finite intersection of semialgebraic sets is semialgebraic, where a semialgebraic set is a set defined by Boolean combination of equalities and inequalities of real polynomials, and hence each $Q_j$ is semialgebraic. This shows that $g_\epsilon$ is piecewise polynomial. As to prove that $g_\epsilon$ is uniquely determined, assume by way of contradiction that there are two different sets $\{f_{1,\epsilon},Ê\ldots, f_{k,\epsilon}\}$ and $\{f'_{1,\epsilon},Ê\ldots, f'_{k',\epsilon}\}$ such that $g_\epsilon=\sum_i \lambda_i\cdot f_{i,\epsilon}=\sum_j \lambda'_j\cdot f'_{j,\epsilon}$. 
Let $P_1,\ldots, P_m$ and $P_1',\ldots, P'_{m'}$ be the semialgebraic sets on which, respectively, the $f_{i,\epsilon}$'s  and the $f'_{j,\epsilon}$'s are monomial. Let $Y$ be a semialgebraic set contained in $P_i\cap P'_j$ for some $i, j$. Thus $g_{\epsilon}$ restricted to $Y$ is polynomial, but $g_{\epsilon}$ is not uniquely displayed on $Y$, since $g_{\epsilon \rest_{Y}} = \sum_i \lambda_i\cdot f_{i,\epsilon \rest_{Y}} =\sum_j \lambda'_j\cdot f'_{j,\epsilon \rest_{Y}}$. But this is contradictory, since polynomial functions have a unique representation on semialgebraic sets. Thus, on each intersection $P_i\cap P'_j$, $g_{\epsilon}$ has a unique representation, and thus it is uniquely determined.
\end{proof}

For every $n\inÊ\mathbb{N}$, every $\epsilon\in \Sigma$ and every $q\in (0,1]\cap\mathbb{Q}$, we denote by $\mathscr{L}_\epsilon^q(n)$ the set of functions obtained by restricting those in $\mathscr{L}_\epsilon(n)$  to $G_\epsilon^q$. 
Further, for every subset $X$ of $[0,1]^n$, we denote by $\mathscr{C}(X)$ the set of all continuos functions from $X$ to $\mathbb{R}$. 


\begin{proposition}\label{fact3}
For every $\epsilon$, for every $q\in (0,1]\cap\mathbb{Q}$ and for every $c\in \mathscr{C}(G_\epsilon^q)$, there is a sequence $g_{1}, g_{2},\ldots\in \mathscr{L}^q_\epsilon(n)$ such that $g_{i}\leq c$ for every $i$, and $\{g_{i}\}$ uniformly converges to $c$.
\end{proposition}
\begin{proof}
Since every $c$ is continuous and defined on a compact set, we can get the claim by the Stone-Weierstrass theorem \cite[\S VIII]{Doob} if we show  that $\mathscr{L}^q_\epsilon(n)$ is a subalgebra of $\mathscr{C}(G_\epsilon^q)$, i.e. it is a vector subspace of $\mathscr{C}(G_\epsilon^q)$ that is closed under multiplication of functions, $\mathscr{L}^q_\epsilon(n)$ contains a non-zero constant function and it separates the points. The first two claims are trivial. Thus, let us show that $\mathscr{L}^q_\epsilon(n)$ separates the points, i.e.,  for every $x,y\in G_\epsilon^q$, if $x\neq y$, then $g(x)\neq g(y)$ for some $g\in \mathscr{L}^q_\epsilon(n)$. Indeed, each monomial function $m$ defined on a subset of $[0,1]^n$ is strictly increasing and hence $m(x)\neq m(y)$ if $x\neq y$, whence the claim is settled. 
\end{proof}

\section{States of free product algebras}\label{sec:states}
Let us start introducing the main notion of our investigation, namely, states of free finitely generated product algebras.
\begin{definition}\label{def-State}
A {\em state of} $\FP(n)$ is a map $s:\FP(n)\to [0,1]$ satisfying the following conditions: 
\begin{itemize}
\item[S1.] $s(1)=1$ and $s(0)=0$,
\item[S2.] $s(f\wedge g)+s(f\vee g)=s(f)+s(g)$,
\item[S3.] If $f\leq g$, then $s(f)\leq s(g)$,
\item[S4.] If $f \neq 0$, then $s(f) = 0$ implies $s(\neg\neg f) = 0$.
\end{itemize}
\end{definition}

The following proposition shows some basic facts about states of free product algebras. Their proofs are straightforward and hence omitted.
\begin{proposition} For any state $s:\FP(n)\to [0,1]$ the following  hold: 
\begin{itemize}
\item[(i)] $s$ restricted to  $\mathscr{B}(\FP(n))$ is a finitely additive probability measure;
\item[(ii)] if $f \land g = 0$, $s(f\lor g)=s(f)+s(g)$. Thus, $s(f \lor \neg f) = s(f) + s(\neg f)$;
\item[(iii)] if $f \lor g = 1$, $s(f \land g)=s(f)+s(g)-1$. Thus, $s(f \leftrightarrow g) = s(f \to g) + s(g \to f) - 1$;
\item[(iv)] $s(\neg f) + s(\neg\neg f) = 1$.
\end{itemize}
\end{proposition}

\begin{remark}
It is worth  pointing out that states of a free product algebra are lattice valuations (axioms S1--S3) as introduced by Birkhoff in \cite{Bir}. However, if we compare Definition \ref{def-State} with the axiomatization of states of an MV-algebra \cite{Mu,FK15}, it is clear that, while for the  MV-case the monoidal operations are directly involved in the axiomatization of states, in our case the unique axiom that involves the multiplicative connectives of product logic is S4.

In the following Proposition \ref{prop:S4} we will prove that S4 can be equivalently substituted by the condition
\begin{itemize}
\item[(S4')] For every $\epsilon \in \Sigma$ and $f \in \FP(n)$, if $f \land p_{\epsilon} \neq 0$, then $s(f \land p_{\epsilon}) = 0$ implies $s(p_{\epsilon}) = 0$,
\end{itemize}
which involves the atoms of $\mathscr{B}(\FP(n))$  and does not make use of the negation connective $\neg$. It is also worth noticing that the condition (S4') quite closely resembles the condition (C4) of \cite{AGM1} where the authors axiomatized the integral on functions of free G\"odel algebras $\mathcal{F}_{\mathbb{G}}(n)$. To be more precise, the condition (C4) (see \cite[\S 2.2]{AGM1}) says the following: for every $x,y,z\in \mathcal{F}_{\mathbb{G}}(n)$ which are either join-irreducible or $0$, if $x<y<z$  and $s(x)=s(y)$, then $s(y)=s(z)$. Turning back to (S4'), if we take $0< f\land p_\epsilon< p_\epsilon$, then we get something similar to (C4). Indeed, if $0=s(0)=s(f\land p_\epsilon)$, we have that $s(p_\epsilon)=0$ as well, whence $s(f\land p_\epsilon)=s(p_\epsilon)$. 
\end{remark}

\begin{proposition}\label{prop:S4}
The axiom {\em S4} is equivalent to
\begin{itemize}
\item[S4'.] For every $\epsilon \in \Sigma$, if $f \land p_{\epsilon} \neq 0$, then $s(f \land p_{\epsilon}) = 0$ implies $s(p_{\epsilon}) = 0$.
\end{itemize}
\end{proposition}
\begin{proof}
$(S4 \Rightarrow  S4')$. If $f \land p_{\epsilon} \neq 0$, then $s(f \land p_\epsilon) = 0$ implies, by $S4$, $s(\neg\neg (f \land p_{\epsilon})) = 0$ as well. But $\neg\neg (f \land p_{\epsilon}) = \neg \neg f \land \neg \neg p_{\epsilon} =  \neg \neg f \land p_{\epsilon}$. Since $f \land p_{\epsilon} \neq 0$,  it turns out that $f_\epsilon \neq 0$, and  by Theorem \ref{thm:repProd}, $f_\epsilon > 0$ on $G_\epsilon$, and hence $\neg\neg f_\epsilon = 1$, thus  $\neg \neg f \land p_{\epsilon} = p_\epsilon$. 
\vspace{.2cm}

$(S4' \Rightarrow S4)$. Let $f \neq 0$. Since 
$$
f = \bigvee_{\epsilon\in \Sigma} f \land p_\epsilon\mbox{, then }s(f) = \sum_{\epsilon\in \Sigma} s(f \land p_\epsilon).
$$ 
If $s(f) = 0$ then we have $s(f \land p_\epsilon) = 0$ for each $\epsilon$. Thus, by $S4'$, if $f \land p_{\epsilon} \neq 0$ we have $0 = s(p_\epsilon) = s(\neg\neg f \land p_\epsilon)$; otherwise, if   $f \land p_{\epsilon} = 0$, then $\neg\neg f \land p_{\epsilon} = 0$ as well.  Therefore, $s(\neg\neg f) = \sum_\epsilon s(\neg\neg f \land p_\epsilon) =0$. 
\end{proof}

In the next subsection we will investigate, as an example, the states of the free $1$-generated product algebra $\FP(1)$ with the aim of exhibiting a first representation for these functional in terms of measures on the dual side. 

\subsection{States of the free $1$-generated product algebra}

The free 1-generated product algebra $\FP(1) \subset [0, 1]^{[0, 1]}$ consists of one-variable functions $f$ of the form $f(x) = t(x)$, where the term $t(x)$ can be either $1$ (the constant function equal to 1), $0$ (the constant function equal to 0), $x$, $\neg x$, $\neg\neg x$, or it belongs to the following set: 
$$\{ x^n \mid n \in \mathbb{N}\} \cup \{x^n \lor \neg x \mid  n \in \mathbb{N}\} . $$ 
The lattice structure of $\FP(1)$ is depicted in Figure \ref{FP1}.

As we recalled in Section \ref{sec:pre}, the Boolean skeleton $\mathscr{B}(\FP(1))$ of $\FP(1)$ coincides with the free Boolean algebra over $1$ generator. Thus, in this case, $\Sigma=\{\epsilon_1, \epsilon_2\}$ and the two atoms of $\mathscr{B}(\FP(1))$ will be denoted by $p_1$ and $p_2$. Therefore, identifying terms with functions, the elements of $\mathscr{B}(\FP(1))$ are $1$, $0$, $\neg x=p_1$ and $\neg\neg x=p_2$ and the  partition $\{G_1, G_2\}$ of $[0, 1]$ is given by $G_1 = \{0\}$ and $G_2 = (0, 1]$.

Then, as it is easy to check, any map $s: \FP(1) \to [0, 1]$ satisfying the following conditions is a state: 

\begin{itemize}
\item $s(1) = 1$, $s(0) = 0$,
\item 
$s(\neg x) + s(\neg\neg x) = 1$,
\item either $s(\neg\neg x) = s(x) = s(x^n) = 0$ for all $n$, or all of them are positive,
\item $s(x^n) \leq s(x^m)$ whenever $n \geq m$,
\item $s(x^n \lor \neg x) = s(x^n) + s(\neg x)$. 
\end{itemize}
\begin{figure} [ht]
\begin{center}

\begin{tikzpicture}

  \tikzstyle{every node}=[draw, circle, fill=white, minimum size=4pt, inner sep=0pt, label distance=1mm]

    \draw (0,0) node (bot) [label=below: $0$] {}; 
    
    \draw (45:1cm) node(x2) [label=right:$x^2$] {}
   
        -- ++(45:1cm)	node (x) [label=right:$x$] {}
        -- ++(45:2cm)	node (notnot) [label=right:$\neg\neg x$] {}
        -- ++(135:2cm)	node (top) [label=above:$1$] {} 
        -- ++(225:2cm)  node (notor) [label=left:$x \lor \neg x$]{}
        -- ++(225:1cm)  node (notor2) [label=left:$x^2 \lor \neg x$]{};

\draw (135:2cm) node(not) [label=left:$\neg x$] {};

\draw (not)--(bot); 

\draw (x)--(notor); 

\draw (x2)--(notor2);

%
%
    \draw [dotted, very thick] (bot)--(x2); 
       \draw [dotted, very thick] (not)--(notor2); 
     
    \draw [dotted, very thick] (45:0.3cm)--(126:2cm);
    
    \draw [dotted, very thick] (45:0.6cm)--(118:2.1cm);
%
    
\end{tikzpicture}

\end{center}
\caption{\label{FP1} The lattice of the free product algebra with one generator $\FP(1)$}
\end{figure}
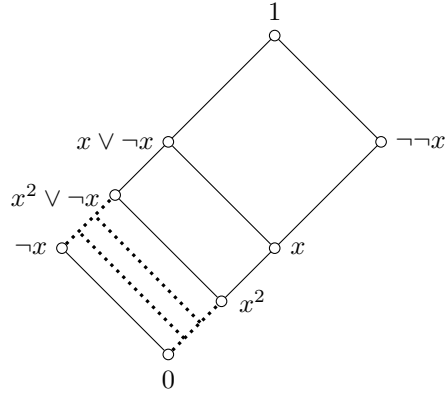


For every $y\in \FP(1)$, let $\langle y \rangle$ denote the principal lattice filter generated by $y$. 
The spectrum, denoted by $\mathfrak{P}$, of prime {\em lattice} filters of $\FP(1)$
, ordered by reverse inclusion, is as in Figure \ref{Spec1}. 
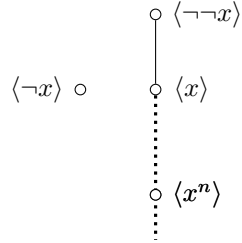
\begin{figure} [ht]
\begin{center}

\begin{tikzpicture}

  \tikzstyle{every node}=[draw, circle, fill=white, minimum size=4pt, inner sep=0pt, label distance=1mm]

    \draw (0,0) node (p1) [label=left: $\langle\neg x\rangle$] {}; 
     \draw (1,0) node (p2) [label=right: $\langle x\rangle$] {}; 
    \draw (1,1) node (p3) [label=right: $\langle\neg\neg x\rangle$] {}; 
         \draw (1,-1.4) node (p4) [ label=right: $\langle x^{n} \rangle$] {};  
    \draw (p2)--(p3); 
          \draw [dotted, very thick] (p2)--(1,-2);
            \draw (1,-1.4) node (p4) [ label=right: $\langle x^{n} \rangle$] {};  
\end{tikzpicture}

\end{center}
\caption{\label{Spec1} The spectral space $\mathfrak{P}$ of the lattice subreduct of the free product algebra with one generator $\FP(1)$}
\end{figure}

Notice that $\mathfrak{P}$ is partially ordered as follows: $\langle \neg x\rangle$ is incompatible with any other element of $\mathfrak{P}$; $\langle \neg\neg x\rangle\geq_\mathfrak{P} \langle x\rangle\geq_\mathfrak{P} \langle x^2 \rangle \geq_\mathfrak{P} \langle x^3 \rangle \geq_\mathfrak{P}\ldots$. 
Priestley duality for bounded distributive lattices \cite{Priestley}, provides us with a lattice isomorphism $R_{(\cdot)}$ between the lattice subreduct of $\FP(1)$ and the lattice of those {\em downsets} of $\mathfrak{P}$, which are clopen with respect to the usual spectral topology. However, since every downset of $\mathfrak{P}$ is clopen, $R_{(\cdot)}$ is onto the whole lattice of downsets of $\mathfrak{P}$. For every $x\in \FP(1)$, it is:
$$
R_x=\{\langle y \rangle\in \mathfrak{P}\mid x\in \langle y\rangle\}.
$$ 
In particular, we have that for instance, $R_{\neg \neg x}=\{\langle y\rangle\in \mathfrak{P}\mid \langle y\rangle\leq_{\mathfrak{P}}\langle\neg\neg x\rangle\}$ and $R_{\neg x\vee x^n}=\{\langle y_1\rangle\in \mathfrak{P}\mid \langle y_1\rangle\leq_{\mathfrak{P}}\langle\neg x\rangle\}\cup \{\langle y_2\rangle\in \mathfrak{P}\mid \langle y_2\rangle\leq_{\mathfrak{P}}\langle x^n\rangle\}$. 
%
%
%
(Fig.  \ref{Spec2} provides examples aimed at clarifying this correspondence). 
The following fact clearly holds:

\begin{fact} \label{claim:dist}
For every $z\in \FP(1)$, if $z=z_1\vee z_2$, then $R_{z}=R_{z_1}\cup R_{z_2}$ and $R_{z_1\wedge z_2}=R_{z_1}\cap R_{z_2}$. 
\end{fact}

%

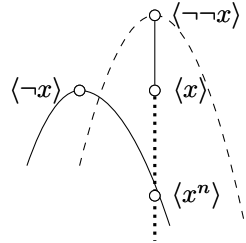
\begin{figure} [ht]
\begin{center}

\begin{tikzpicture}
\tikzstyle{every node}=[draw, circle, fill=white, minimum size=4pt, inner sep=0pt, label distance=1mm]

    \draw (0,0) node (p1) [draw, circle, label=left: $\langle\neg x\rangle$] {}; 
     \draw (1,0) node (p2) [draw, circle,label=right: $\langle x\rangle$] {}; 
    \draw (1,1) node (p3) [draw, circle,label=right: $\langle\neg\neg x\rangle$] {}; 
     \draw (1,-1.4) node (p4) [label=right: $\langle x^{n}\rangle$] {}; 
    \draw (p2)--(p3); 
    \draw (-0.7,-1) parabola bend (0,0)
    (1.2, -1.8);
        \draw[dashed] (0,-1) parabola bend (1,1)
    (2.2, -1.8);
        \draw [dotted, very thick] (p2)--(1,-2);
          \draw (1,-1.4) node (p4) [ label=right: $\langle x^{n} \rangle$] {};  
           \draw (0,0) node (p1) [draw, circle, label=left: $\langle\neg x\rangle$] {}; 
     \draw (1,0) node (p2) [draw, circle,label=right: $\langle x\rangle$] {}; 
\draw (1,1) node (p3) [draw, circle,label=right: $\langle\neg\neg x\rangle$] {}; 
\end{tikzpicture}

\end{center}
\caption{\label{Spec2} This figure shows 
the downsets $R_{\neg\neg x}$ (dashed parabola) and  $R_{\neg x\vee x^n}$ (continuous parabola)}
\end{figure}

With such a representation in mind, let $s$ be a state of $\FP(1)$ and let us define a $[0,1]$-valued function $d_s$ on  $\mathfrak{P}$ in the following way:
\begin{itemize}
\item[i.] $d_s(\langle \neg x\rangle)=s(\neg x)$,
\item[ii.] $d_s(\langle\neg\neg x\rangle)=s(\neg\neg x)-s(x)$, 
\item[iii.] for every $n\in \mathbb{N}$, $d_s(\langle x^n \rangle)=s(x^n)-s(x^{n+1})$.
\end{itemize}
First of all, let us show that $d_s$ is a {\em (discrete) probability distribution}
on $\mathfrak{P}$, indeed: 
$$
\begin{array}{lll}
\sum_{\langle y\rangle\in \mathfrak{P}}d_s(\langle y\rangle)&=&d_s(\langle \neg x\rangle)+d_s(\langle \neg\neg x\rangle)+d_s(\langle x\rangle)+\sum_{n>1} d_s(\langle x^n \rangle) \\
&=& s(\neg x)+s(\neg\neg x)-s(x)+s(x)-s(x^2)+\sum_{n>1} d_s(\langle x^n \rangle)\\
&=& s(\neg x)+s(\neg\neg x)-s(x^2)+s(x^2)-s(x^3)+\sum_{n>2} d_s(\langle x^n \rangle)\\
&=& s(\neg x)+s(\neg\neg x)-s(x^3)+\sum_{n>2} d_s(\langle x^n \rangle)\\
&=&\dots\\
&=& s(\neg x)+s(\neg\neg x)\\
&=& s(\neg x\vee \neg\neg x)\\
&=& s(1)\\
&=&1.
\end{array}
$$
Hence, every state of $\FP(1)$ determines a distribution function $d_s$ on $\mathfrak{P}$. Moreover, notice that the condition (S4) of Definition \ref{def-State} forces $d_s$ to satisfy the following further condition:\\
\begin{itemize}
\item[(D)] If $d_s(\langle y\rangle)=0$, then $d_s(\langle y'\rangle)=0$ for every $\langle y'\rangle\geq_{\mathfrak{P}} \langle y\rangle$.\\
\end{itemize}

\noindent Conversely, let $d:\mathfrak{P}\to[0,1]$ be a distribution satisfying (D) and define $s_d:\FP(1)\to[0,1]$ by the following stipulation: for every $z\in \FP(1)$, 
\begin{equation}\label{eq:statesDens}
s_d(z)=\sum_{\langle y\rangle\in R_z} d(\langle y\rangle).
\end{equation}
Let us show that $s_d$ is a state of $\FP(1)$. Obviously $s_d(1)=1$ and $s_d(0)=0$. As to prove additivity, let $z_1, z_2$ in $\FP(1)$. From Fact \ref{claim:dist}, $R_{z_1\vee z_2}=R_{z_1}\cup R_{z_2}$ and $R_{z_1\wedge z_2}=R_{z_1}\cap R_{z_2}$. Thus, 
$$
\begin{array}{lll}
s_d(z_1\vee z_2)&=& \sum_{\langle y\rangle\in R_{z_1\vee z_2}}d(\langle y\rangle)\\
&=& \sum_{\langle y\rangle\in R_{z_1}\cup R_{z_2}}d(\langle y\rangle)\\
&=&  \sum_{\langle y\rangle\in R_{z_1}} d(\langle y\rangle) + \sum_{\langle y'\rangle\in R_{z_2}}d(\langle y'\rangle)-\sum_{\langle y''\rangle\in R_{z_1}\cap R_{z_2}} d(\langle y''\rangle)\\
&=&  \sum_{\langle y\rangle\in R_{z_1}} d(\langle y\rangle) + \sum_{\langle y'\rangle\in R_{z_2}}d(\langle y'\rangle)-\sum_{\langle y''\rangle\in R_{z_1\wedge z_2}} d(\langle y''\rangle)\\
&=&  s_d(z_1)+s_d(z_2)-s_d(z_1\wedge z_2).
\end{array}
$$
The monotonicity of $s_d$  can be proved  in a similar manner observing that $z_1\leq z_2$ iff $R_{z_1}\subseteq R_{z_2}$.

Let us finally prove that (S4) is satisfied. The two atoms of $\mathscr{B}(\FP(1))$ are $\neg x$ and $\neg\neg x$ and for every $y\in \FP(1)$, either $y\wedge \neg x= \neg x$ if $y=\neg x$, or $y\wedge\neg x=0$ and in this case (S4) is trivially satisfied. As for $\neg\neg x$, let $y\in \FP(1)$ such that $y\wedge\neg \neg x\neq0$. Then, as it is evident from Figure \ref{FP1} (and skipping the trivial cases of $y=1$ and $y=\neg \neg x$) either $y=x^n$ for $n\geq 1$, or $y=\neg x\vee x^n$, for $n\geq 1$. In both cases of $y=x^n$ or $y=\neg x\vee x^n$, $\neg\neg x\wedge y=x^n$. Thus, if $s_d(y\wedge \neg\neg x)=s(x^n)=\sum_{\langle t\rangle\in R_{x^n}}d(\langle t\rangle)=0$, $d(\langle t\rangle)=0$ for all $\langle t\rangle\in R_{x^n}$. Notice that for any other $\langle a\rangle\in R_{\neg\neg x}$, one has $\langle a\rangle\geq_{\mathfrak{P}} \langle t\rangle$ for each $\langle t\rangle\in R_{x^n}$, whence by (D), $d(\langle a\rangle)=0$ ensuring that $s_d(\neg\neg x)=\sum_{\langle a\rangle\in R_{\neg\neg x}}d(\langle a\rangle)=0$. 

Thus, the following holds.
\begin{proposition}
There is a one-one correspondence between the set of states of $\FP(1)$ and the set of distribution functions on $\mathfrak{P}$ that satisfy (D). 
\end{proposition}

\begin{proof}
From what we showed above, we can define a map that associates a distribution $d_{s}$ to each state $s$. Thus, it is sufficient to prove that the map is injective, since surjectivity is obvious from (\ref{eq:statesDens}). To this end, let $s_1\neq s_2$ be two states of $\FP(1)$. Thus, there is a $y \in \FP(1)$ such that $s_1(y)\neq s_2(y)$. 
Now, if $y$ is one among $\{\neg x, \neg\neg x\} \cup \{x^{n} \mid n \in \mathbb{N}\}$, it is clear that $d_{s_{1}}(\langle y \rangle) \neq d_{s_{2}}(\langle y \rangle)$. If $y$ is of the kind $x^{n} \lor \neg x$, for $n \in \mathbb{N}$, since $\neg x \land x^{n} = 0$, using (S2) we obtain that $s_{1}(\neg x) + s_{1}(x^{n}) \neq s_{2}(\neg x) + s_{2}(x^{n})$. Therefore, either $s_1(\neg x)\neq s_2(\neg x)$, and thus $d_{s_{1}}(\langle \neg x \rangle) \neq d_{s_{2}}(\langle \neg x \rangle)$, or $s_1(x^{n})\neq s_2(x^{n})$, and thus $d_{s_{1}}(\langle x^{n} \rangle) \neq d_{s_{2}}(\langle x^{n} \rangle)$, which settles the proof.
\end{proof}

\section{Integral representation}\label{sec:integral}

As we recalled in Section \ref{sec:pre}, product functions are not continuous, thus, unlike the case of (free) MV-algebras, an integral representation for states cannot be obtained by directly applying Riesz representation theorem for linear and monotone functionals.\footnote{Recall that Riesz theorem says that, if X is a locally compact Hausdorff space, then for any positive linear functional $I$ on the space $\mathscr{C}_C(X)$ of continuous functions with compact support on $X$, there is a unique regular Borel measure $\mu$ on $X$ such that $I(f) = \int f {\rm d}\mu$, see e.g. \cite[Theorem 2.14]{Ru}.}
%
However, 
the {\em finite} partition $\{G_\epsilon\mid \epsilon\in \Sigma\}$ of $[0,1]^n$ is made of $\sigma$-locally compact sets (Remark \ref{rem:nested}) upon which the restriction $f_\epsilon$ of each product function $f$ is continuous. 
In this setting, we will suitably extend states to real-valued, positive, monotone and linear operators acting on all  continuous functions on the restricted compact domain.  Only then we will in position 
to apply Reisz representation theorem to obtain Borel measures over each $G_{\varepsilon}^{q}$, in such a way that the Lebesgue integral with respect to these measures will act exactly like our  properly restricted functionals. Finally, we will suitably extend the measures obtained by Riesz theorem first to measures on every $G_{\varepsilon}$, and secondly to a measure $\mu$ on the Borel subsets of real unit cube $[0,1]^{n}$. We will hence prove that the Lebesgue integral with respect to $\mu$ behaves like the state $s$ over the functions of $\mathcal{F}_{\mathbb{P}}(n)$.

Thanks to Proposition \ref{prop:unique}, given a state $s$ of $\FP(n)$, for every $\epsilon\in \Sigma$ we can define  a map $\tau_\epsilon:\mathscr{L}_\epsilon(n)\to \mathbb{R}$ in the  following way. Indeed, by Proposition \ref{prop:unique}, every $g\in\mathscr{L}_\epsilon(n) \setminus\{0\}$ is uniquely represented as a linear combination $\sum_{i=1}^k\lambda_i\cdot f_{i,\epsilon}$ for (uniquely determined) non-zero parameters $\lambda_1,\ldots,\lambda_{k}$ and distinct $f_{1,\epsilon},\ldots,f_{k,\epsilon}\in \mathscr{P}_\epsilon(n)$. Thus, we can properly define: 

\begin{equation}\label{eq:tauEpsilon}
\tau_\epsilon(g)= \tau_\epsilon\left(\sum_{i=1}^k\lambda_i\cdot f_{i,\epsilon}\right)=\sum_{\{i\mid s(f_{i,\epsilon}\wedge p_\epsilon)>0\}} \lambda_i\cdot \frac{s(f_{i,\epsilon}\wedge p_\epsilon)}{s(p_\epsilon)}.
\end{equation}

%

Notice that if, for some $i$, $s(f_{i,\epsilon}\wedge p_\epsilon)=0$,  (S4') ensures $s(p_{\epsilon})=0$ and hence $s(f_{j,\epsilon}\wedge p_\epsilon)=0$ as well for any other $j \neq i$, and in such a case, \eqref{eq:tauEpsilon} yields $\tau_\epsilon(g)= 0$ with the proviso that the empty sum is taken to be 0. 

The definition of $\tau_\epsilon$ is completed by putting $\tau_\epsilon(0) = 0$.

\begin{proposition}\label{prop:linear}
For every state $s$ of $\FP(n)$ and for every $\epsilon$, $\tau_\epsilon$ is a linear and monotone map. 
\end{proposition}
\begin{proof}
Linearity follows by the very definition of $\tau_\epsilon$. 
As for the monotonicity of $\tau_{\epsilon}$ let $0<g\leq g'$ with $g=\sum_{i=1}^{k} \lambda_{i}\cdot f_{i,\epsilon}$ and $g'=\sum_{i=1}^{k'} \lambda'_{i}\cdot f'_{i,\epsilon}$ as given by Proposition \ref{prop:unique}. Then,  $\tau_{\epsilon}(g)>0$ implies $\tau_{\epsilon}(g')\geq \tau_{\epsilon}(g)$ by definition and the monotonicity of $s$. On the other hand, if $\tau_{\epsilon}(g)=0$, then $s(f_{i,\epsilon}\wedge p_\epsilon)=0$ for every $i$ whence, by (S4'), $s(p_{\epsilon})=0$. Therefore, $s(f'_{i,\epsilon}\wedge p_\epsilon)=0$ by monotonicity of $s$ thus, $\tau_{\epsilon}(g')=0$.
\end{proof}

Now, for every  $q\in (0,1]\cap\mathbb{Q}$, let $\tau_\epsilon^q: \mathscr{L}^q_\epsilon(n)\to \mathbb{R}$ be defined as follows: for every $g\in \mathscr{L}^q_\epsilon(n)$, 
$$
\tau_\epsilon^q(g)=\inf\{\tau_\epsilon(g')\mid g'\in \mathscr{L}_\epsilon(n),\; g'_{\rest{G_\epsilon^q}}=g\}.
$$

\begin{proposition}\label{lemma:tauQ}
$\tau_{\epsilon}^q$ is a linear and monotone functional over $\mathscr{L}^q_\epsilon(n)$.
Moreover, if $q_2 \leq q_1$, $g\in \mathscr{L}^{q_1}_\epsilon(n)$, and $g'\in \mathscr{L}^{q_2}_\epsilon(n)$ extends $g$, then  $\tau_{\epsilon}^{q_{1}} (g) \leq \tau_{\epsilon}^{q_{2}}(g')$.
\end{proposition}
\begin{proof}
See Appendix A.
\end{proof}
Now, we want to extend $\tau_\epsilon^q$ to a linear and monotone functional on the set $\mathscr{C}(G_\epsilon^q)$ of real-valued continuous functions over $G_\epsilon^q$. For every $c\in\mathscr{C}(G_\epsilon^q)$, let $Seq(c)$ be the set of countable increasing sequences $\overline{g}=\{g_i\}_{i\in \mathbb{N}}$ of elements in $\mathscr{L}^q_\epsilon(n)$ uniformly converging to $c$, in symbols, $\overline{g} \nnearrow\! c$.

 
 Thus, for every $c\in \mathscr{C}(G_\epsilon^q)$  and for every $\overline{g}\in Seq(c)$ we first define
 $$
 \sigma_{\overline{g}}(c)=\bigvee_{i\in\mathbb{N}}\tau_\epsilon^q(g_i), 
 $$
 and  finally we put
 \begin{equation}\label{eq:sigmaXQ}
 \sigma_\epsilon^q(c)=\bigvee_{\overline{g}\in Seq(c)}\sigma_{\overline{g}}(c).
 \end{equation}
 
 \begin{lemma}\label{lemmaSigmaEpsilon}
For every $\epsilon\in \Sigma$ and  every $q\in (0,1]\cap\mathbb{Q}$, $\sigma^q_{\epsilon}$ is a positive, monotone and linear functional. Moreover $\sigma^q_\epsilon$ extends $\tau^q_\epsilon$ on $\mathscr{L}^q_\epsilon(n)$. 
\end{lemma}

\begin{proof}
See Appendix A.
\end{proof}

The previous Lemma \ref{lemmaSigmaEpsilon} has the following immediate consequence.
\begin{theorem}\label{cor:int}
For every $\epsilon \in \Sigma$ and every rational $q$, there is a unique regular Borel measure $\mu_\epsilon^{q}$ such that, for any $c \in  \mathscr{C}(G^q_\epsilon) $, 
$$\sigma_\epsilon^q(c)  = \int_{G^q_\epsilon} c \:{\rm d}\mu_{\epsilon}^q.$$
In particular, for all $g\in \mathscr{L}^q_\epsilon(n)$,
$$
\tau^q_\epsilon(g)=\int_{G^q_\epsilon} g \:{\rm d}\mu_{\epsilon}^q.
$$
\end{theorem}
\begin{proof}
From Lemma \ref{lemmaSigmaEpsilon}, $\sigma_\epsilon^q$ is a (positive) linear functional over $\mathscr{C}(G^q_\epsilon)$, with $G^q_\epsilon$ being compact, thus the the first part of the claim  follows from Riesz representation theorem \cite[Theorem 2.14]{Ru}. The last part of the claim, finally follows from the last part of Lemma \ref{lemmaSigmaEpsilon}.
\end{proof}
With respect to the notation used in the previous Theorem \ref{cor:int},  the following  lemma holds.
\begin{lemma}\label{lemma:intTau}
For every Borel subset $B$ of $G_\epsilon$ and for every $q\in (0,1]\cap\mathbb{Q}$, if $q'\leq q$, $\mu_\epsilon^q(B\cap G_\epsilon^q)=\mu_\epsilon^{q'}(B\cap G_\epsilon^{q})$.
\end{lemma}
\begin{proof}
Let us write $Z=B\cap G_\epsilon^q$. First of all, notice that $\mu_\epsilon^q(Z)=\inf\{\sigma_{\epsilon}^q(c)\mid c\in \mathscr{C}(G_\epsilon^q), \; c\geq \chi_Z\}$ and $\mu_\epsilon^{q'}(Z)=\inf\{\sigma_{\epsilon}^{q'}(c')\mid c'\in \mathscr{C}(G_\epsilon^{q'}), \; c'\geq \chi_Z\}$ (where $\chi_Z$ denotes the characteristic function of $Z$). 

Since for each $r\in (0,1]\cap\mathbb{Q}$, $\mathscr{L}_\epsilon^r(n)$ is dense in $\mathscr{C}(G_\epsilon^r)$, we can safely write, for $q^*$ either being $q$ or $q'$, 
$$
\mu_\epsilon^{q^*}(Z)=\inf\{\sigma_{\epsilon}^{q^*}(g)\mid g\in \mathscr{L}_\epsilon^{q^*}(n), \; g\geq \chi_Z\},
$$
whence, from Lemma \ref{lemmaSigmaEpsilon}, 
$$
\mu_\epsilon^{q^*}(Z)=\inf\{\tau_{\epsilon}^{q^*}(g)\mid g\in \mathscr{L}_\epsilon^{q^*}(n), \; g\geq \chi_Z\}.
$$
Let us now define
$$
\Delta=\{g\in \mathscr{L}_\epsilon^{q}(n), \; g\geq \chi_Z\}\mbox{ and } \Delta'=\{g'\in \mathscr{L}_\epsilon^{q'}(n), \; g'\geq \chi_Z\}
$$
and 
$$
I=\{f\in \mathscr{L}_\epsilon(n)\mid f_{\rest G_\epsilon^q}\in \Delta\}\mbox{ and }I'=\{h\in \mathscr{L}_\epsilon(n)\mid h_{\rest G_\epsilon^{q'}}\in \Delta'\}.
$$
Clearly $I=I'$. Indeed, if $f\in I$, then $f\in \Delta(\epsilon)$, whence $f_{\rest G_\epsilon^{q'}}\in \mathscr{L}_\epsilon^{q'}(n)$ and $f_{\rest G_\epsilon^{q'}}\geq \chi_Z$ because $q'\leq q$, whence $Z=B\cap G_\epsilon^q\subseteq B\cap G_\epsilon^{q'}$. Conversely, if $h\in I'$, $h_{\rest G_\epsilon^q}\geq \chi_Z$ again because $q'\leq q$. Then, for every $g\in \Delta$, there is a $g'\in \Delta'$ such that $\tau_\epsilon^q(g)=\tau_{\epsilon}^{q'}(g')$ and vice versa. Thus, by the very definition of $\tau_\epsilon^q$, the claim is settled.
\end{proof}
Now, recalling  Remark \ref{rem:nested},  for every $\epsilon$ and for $q_1\geq q_2$, we have $G_\epsilon^{q_1}\subseteq G_\epsilon^{q_2}$. Thus, the following is an immediate consequence of the above result.
\begin{corollary}\label{cor:BorelMeasure}
If $B$ is a Borel subset of $G_\epsilon^q$ for some $q$, then for all $q'\leq q$, $\mu_\epsilon^q(B)=\mu_\epsilon^{q'}(B)$. 
\end{corollary}

We can now establish an integral representation for the linear and monotone functionals $\tau_\epsilon$ on $\mathscr{L}_\epsilon(n)$. 

\begin{lemma}\label{lemma:riesz}
For every $\epsilon \in \Sigma$, there is a 
Borel probability measure $\mu_\epsilon$ on the Borel subsets of $G_\epsilon$ such that, for every $g\in \mathscr{L}_\epsilon(n)$,
\begin{equation}\label{eqMuEpsilon}
\tau_\epsilon(g)=\int_{G_\epsilon}g\;{\rm d}\mu_\epsilon.
\end{equation}
\end{lemma}
\begin{proof}
See Appendix A.
\end{proof}

Finally, based on the previous results, next theorem provides an integral representation for states of product logic functions.

\begin{theorem}[Integral representation]\label{int-repr}
For every state $s$ of $\FP(n)$ there is a unique regular Borel probability measure $\mu$ such that 
$$
s(f)=\int_{[0,1]^n} f\; {\rm d}\mu.
$$
\end{theorem}
\begin{proof}
For every $f\in \FP(n)$ and for every $\epsilon\in \Sigma$, 
$$
f=\bigvee_{\epsilon\in \Sigma} (f_{\epsilon}\wedge p_\epsilon)=\sum_{\epsilon\in \Sigma} (f_{\epsilon}\wedge p_\epsilon).
$$ 
Moreover, for distinct $\epsilon_1, \epsilon_2$, $(f_{\epsilon_1}\wedge p_{\epsilon_1})\wedge (f_{\epsilon_2}\wedge p_{\epsilon_2})=0$ (since $G_{\epsilon_1}\cap G_{\epsilon_2}=\emptyset$), whence $s((f_{\epsilon_1}\wedge p_{\epsilon_1})\wedge (f_{\epsilon_2}\wedge p_{\epsilon_2}))=0$. Thus, by axiom S2 
\begin{equation}\label{int1}
s(f)=s\left(\bigvee_{\epsilon\in \Sigma} (f_{\epsilon}\wedge p_\epsilon)\right)=
\sum_{\epsilon\in \Sigma}s(f_{\epsilon}\wedge p_\epsilon).
\end{equation}
Now, from the definition of $\tau_\epsilon$, Proposition \ref{prop:linear} and Lemma \ref{lemma:riesz} it follows that
\begin{equation}\label{int2}
s(f_{\epsilon}\wedge p_\epsilon)=s(p_\epsilon)\cdot \tau_\epsilon(f_{\epsilon})=s(p_\epsilon)\cdot\int_{G_\epsilon} f_{\epsilon}\;{\rm d}\mu_\epsilon
\end{equation}
for a Borel measure $\mu_\epsilon$ on the Borel subsets of $G_\epsilon$.

Let hence define $\mu$ on the Borel subsets of $[0,1]^n$ by the following stipulation: for every $X$ Borel subset of $[0,1]^n$,
$$
\mu(X)=\sum_{\epsilon\in \Sigma} s(p_\epsilon)\cdot \mu_\epsilon(X\cap G_\epsilon).
$$
Since $\sum_{\epsilon\in \Sigma}s(p_\epsilon)=s(\bigvee_{\epsilon\in \Sigma}p_\epsilon)=s(\top)=1$,  $\mu$ is  a convex combination of the  $\mu_\epsilon$'s. Moreover $\mu$ is defined for every $X$ since $G_\epsilon$ is a Borel subset of $[0,1]^n$ (recall Lemma \ref{lemma:borel}), whence $G_\epsilon\cap X$ is Borel as well. 
Thus, from (\ref{int1}) and (\ref{int2}),
$$
\begin{array}{lll}
s(f)&=&\displaystyle\sum_{\epsilon\in \Sigma} s(p_\epsilon)\cdot \tau_\epsilon(f_{\epsilon})\\
&=&\displaystyle\sum_{\epsilon\in \Sigma} \left(s(p_\epsilon)\cdot \displaystyle\int_{G_\epsilon}f_{\epsilon}\;{\rm d}\mu_\epsilon\right)\\
&=&\displaystyle\int_{\bigcup_{\epsilon\in \Sigma}G_\epsilon} \sum_{\epsilon\in \Sigma} f_\epsilon\;{\rm d}\left( s(p_\epsilon)\cdot \mu_\epsilon\right)\\
&=&\displaystyle\int_{[0,1]^n} f\;{\rm d}\mu.
\end{array}
$$
It is left to show that $\mu$ is unique. Suppose, by way of contradiction, that for a state $s$ there are two distinct regular Borel measures $\mu_1$ and $\mu_2$ such that, for every $f \in \FP(n)$, $s(f)=\int_{[0,1]^n} f\; {\rm d}\mu_1 = \int_{[0,1]^n} f\; {\rm d}\mu_2$. 

If  $\mu_1 \neq \mu_2$, then there must exist a  $c\in \mathscr{C}([0,1]^n)$  such that $\int_{[0,1]^n} c\; {\rm d}\mu_1 \neq \int_{[0,1]^n} c\; {\rm d}\mu_2$. Since  $[0,1]^n = \bigcup_\epsilon G_\epsilon$, there is an $\epsilon$ such that $\int_{G_\epsilon} c\; {\rm d}\mu_1 \neq \int_{G_\epsilon} c\; {\rm d}\mu_2$. Now, since $G_\epsilon = \bigcup_{q > 0} G^q_\epsilon$, we have $\lim_{q \to 0} \int_{G^q_\epsilon} c\; {\rm d}\mu_1 \neq \lim_{q \to 0} \int_{G^q_\epsilon} c\; {\rm d}\mu_2$. 

Without loss of generality, assume  $\lim_{q \to 0} \int_{G^q_\epsilon} c\; {\rm d}\mu_1 < \lim_{q \to 0} \int_{G^q_\epsilon} c\; {\rm d}\mu_2$. Hence, there is $q$ such that for every $q'$, $\int_{G^{q'}_\epsilon} c\; {\rm d}\mu_1 <  \int_{G^q_\epsilon} c\; {\rm d}\mu_2$. In particular,  $\int_{G^{q}_\epsilon} c\; {\rm d}\mu_1 <  \int_{G^q_\epsilon} c\; {\rm d}\mu_2$. 

Now, over $G^q_\epsilon$, $c$ is the limit of an increasing sequence $\{g_k\}_{k\in \mathbb{N}} \subseteq \mathscr{L}^q_\epsilon(n)$, and by the continuity of the integral, 
$$\sup_k \int_{G^{q}_\epsilon} g_k\; {\rm d}\mu_1 <  \sup_k \int_{G^q_\epsilon} g_k\; {\rm d}\mu_2$$
So there is $k$ such that, for every $k'$, 
$$\int_{G^{q}_\epsilon} g_{k'}\; {\rm d}\mu_1 <  \int_{G^q_\epsilon} g_k\; {\rm d}\mu_2$$
in particular, 
$$\int_{G^{q}_\epsilon} g_{k}\; {\rm d}\mu_1 <  \int_{G^q_\epsilon} g_k\; {\rm d}\mu_2.$$

But $g_k$ is the restriction of a function $g \in \mathscr{L}_\epsilon(n)$ on $G^{q}_\epsilon$ and hence $g$ is of the form $\sum_i \lambda_i f_i$, with $f_i \in  \mathscr{P}_\epsilon(n)$.
Thus, we can apply Theorem \ref{cor:int} and, since the $\tau_\epsilon^q$'s are uniquely determined from the state $s$, we get that $\tau^q_\epsilon(g_k)  = \int_{G^{q}_\epsilon} g_{k}\; {\rm d}\mu_1$ and $\tau^q_\epsilon(g_k)  = \int_{G^q_\epsilon} g_k\; {\rm d}\mu_2$, whence:
$$\tau^q_\epsilon(g_k) < \tau^q_\epsilon(g_k),$$
which is a contradiction. 
\end{proof}

We shall now see that the converse also holds.
\begin{theorem}\label{inv-rep}
For every regular Borel probability measure $\mu:\mathcal{B}([0,1]^n)\to[0,1]$, the function $s: \FP(n) \to [0,1]$ defined as 
$$
s(f)=\int_{[0,1]^n} f\; {\rm d}\mu.
$$
is a state of $\FP(n)$.
\end{theorem} 
\begin{proof}
First we observe that for each $f \in \FP(n)$, $\int_{[0,1]^n} f\; {\rm d}\mu \in [0,1]$, since $\mu$ is normalized to 1 and the functions of the free product algebra take values in $[0,1]$. In order to prove that $s$ is a state, we need to show that the integral of product functions satisfy the properties S1-S4:
\begin{itemize}
\item[(S1)] $\int_{[0,1]^n} {\bf 0} \; {\rm d}\mu = 0$ and $\int_{[0,1]^n} {\bf 1} \; {\rm d}\mu = 1$, where ${\bf 0}$ and ${\bf 1}$ are respectively the functions constantly equal to 0 and 1.
\item[(S2)] $\int_{[0,1]^n} (f \land g)\; {\rm d}\mu + \int_{[0,1]^n} (f \lor g)\; {\rm d}\mu = \int_{[0,1]^n} f \; {\rm d}\mu + \int_{[0,1]^n} g\; {\rm d}\mu$, for each $f, g \in \FP(n)$.
\item[(S3)] If $f, g \in \FP(n)$ are such that $f \leq g$, then $\int_{[0,1]^n} f\; {\rm d}\mu \leq \int_{[0,1]^n} g\; {\rm d}\mu$.
\item[(S4)] For every $\epsilon \in \Sigma$, for any $f \in \FP(n)$ such that $f \land p_{\epsilon}$ is non-zero, $\int_{[0,1]^n} f \land p_{\epsilon}\; {\rm d}\mu = 0$ implies $\int_{[0,1]^n} p_{\epsilon}\; {\rm d}\mu = 0$.
\end{itemize}
Properties (S1) and (S3) are well-known properties of the integral with respect to probability measures. 
About property (S2), it is not difficult to realize that, since the operations are defined pointwise, it holds that $f + g = \min (f,g) + \max (f,g)$, which settles the proof. In order to prove (S4), we shall observe that 
$$
\int_{[0,1]^n} (f \land p_{\epsilon})\; {\rm d}\mu = \int_{[0,1]^n \setminus G_{\epsilon}} (f \land p_{\epsilon})\; {\rm d}\mu + \int_{G_{\epsilon}} (f \land p_{\epsilon})\; {\rm d}\mu.
$$ 
The first integral is $0$ since the function $f \land p_{\epsilon}$ is $0$ outside $G_{\epsilon}$. Thus, if $\int_{[0,1]^n}(f\land p_\epsilon)\; {\rm d}\mu=0$ the second one must be $0$ as well, and since $f \land p_{\epsilon}$ is strictly positive over $G_{\epsilon}$ (if it is $0$ in one point, it is $0$ in the whole $G_{\epsilon}$,  \cite[Lemma 3.2.3]{ABG}) then it must be $\mu(G_{\epsilon}) = 0$, whence $\int_{[0,1]^n} p_{\epsilon}\; {\rm d}\mu = 0$.
\end{proof}
 Therefore, our main result can be stated in the following concise way.
\begin{corollary}\label{cor:main}
For every $n\in \mathbb{N}$, and for every map $s:\FP(n)\to[0,1]$ the following are equivalent:
\begin{itemize}
\item[(1)] $s$ is a state,
\item[(2)] there is a unique regular Borel measure $\mu$ such that, for every $f\in \FP(n)$,
$$
s(f)=\int_{[0,1]^n} f\; {\rm d}\mu.
$$
\end{itemize}
\end{corollary}

\section{The state space and its extremal points}\label{sec:extremal}

In this section we shall prove that states of $\mathcal{F}_{\mathbb{P}}(n)$ are actually convex combinations of  product logic valuations. The idea is to show first that the state space is convex and compact and hence, by Krein-Milman theorem\footnote{
Recall that Krein-Milman theorem  says that if $X$ is a locally convex topological vector space and $K$ is a compact convex subset of $X$,  then $K$ coincides with the closure of the convex hull of its extreme points, see e.g. \cite[Theorem 5.17]{Goodearl}.}, every state is in the closure of convex combination of extremal. Second, we show that the extremal states coincide with  product logic valuations, i.e.\ homomorphisms of $\FP(n)$ into $[0,1]_\Pi$. 

Let $n$ be any positive integer. Let us denote by $\mathcal{H}(n)$  the set of homomorphisms of $\FP(n)$ to the product algebra $[0,1]_\Pi$;
$\mathcal{S}(n)$ stands for the set of  states of $\FP(n)$; $\mathcal{M}(n)$ denotes the set of  regular Borel probability measures on ${\bf B}([0,1]^n)$, the $\sigma$-algebra of Borel subsets of $[0,1]^n$. 
\begin{proposition}\label{homoprod}
For every $n\in \mathbb{N}$, there is a bijection between $\mathcal{H}(n)$ and $[0,1]^{n}$.
\end{proposition}
\begin{proof}
Let $\varphi : [0,1]^{n} \to [0,1]^{\FP(n)}$ be the map that associates to every $x \in [0,1]^{n}$ the function $\varphi_{x}:f \mapsto f(x)$, for every $f \in \FP(n)$.
Clearly, $\varphi_{x}$ is a homomorphism, and it is easy to see that if $x_{1} \neq x_{2}$ then $\varphi_{x_{1}} \neq \varphi_{x_{2}}$. Moreover, every homomorphism $h$ is such that $h = \varphi_{x}$, for some $x \in [0,1]^{n}$. Indeed, let $x=(h(\pi_{1}), \ldots, h(\pi_{n}))$, where $\pi_{i}$ denotes the $i$-th projection. Moreover, for every $f\in \FP(n)$ there is a term $t_f$ such that $f=t_f[\pi_1,\ldots,\pi_n]$. 
Thus $h(f)=h(t_f[\pi_1,\ldots,\pi_n])=t_f[h(\pi_1),\ldots, h(\pi_n)]=f(x)=\varphi_x(f)$. 
\end{proof}
It is quite obvious that $\mathcal{S}(n)$ and $\mathcal{M}(n)$ are convex subsets of $[0,1]^{\FP(n)}$ and $[0,1]^{{\bf B}([0,1]^n)}$ respectively. Furthermore, $\mathcal{M}(n)$ is clearly compact with respect to the subspace product topology. 
As for $\mathcal{S}(n)$, let us prove that it is closed, whence compact.

\begin{proposition}\label{prop:closed}
$\mathcal{S}(n)$ is closed in the Tychonoff cube $[0,1]^{\FP(n)}$.
\end{proposition}
\begin{proof}
See Appendix A.
\end{proof}

By Proposition \ref{prop:closed} above and Krein-Milman theorem, $\mathcal{S}(n)$ and $\mathcal{M}(n)$ are generated by their extremal points. It is well-known that the extremal points of $\mathcal{M}(n)$ are Dirac measures, i.e. those maps $\delta_x:{\bf B}([0,1]^n)\to \{0,1\}$,  for each $x\in [0,1]^n$, such that $\delta_x(B)=1$ iff  $x\in B$ and $\delta_x(B)=0$ otherwise. 

Let us consider the map
$$
\delta: \mathcal{S}(n)\to\mathcal{M}(n)
$$ 
which associates, to each state $s\in \mathcal{S}(n)$ the unique regular Borel measure $\mu\in \mathcal{M}(n)$ provided by Theorem \ref{int-repr}, such that for every $f \in \FP(n)$, $s(f)=\int_{[0,1]^n} f\; {\rm d}\mu.$
The following holds:
\begin{proposition}\label{prop:affine}
For every $n\in \mathbb{N}$, the map $\delta:\mathcal{S}(n)\to \mathcal{M}(n)$ defined as above is bijective and affine.
\end{proposition}
\begin{proof}
Injectivity follows from Theorem \ref{int-repr}, and surjectivity from Theorem \ref{inv-rep}.
In order to prove that $\delta$ is affine, let us suppose that $s = \lambda s_{1} + (1 - \lambda)s_{2}$, with $\lambda \in [0,1]$. Then we have, for every $f \in \FP(n)$, 
$$
\begin{array}{lll}
s(f) &= &\lambda s_{1}(f) + (1 - \lambda)s_{2}(f)\\
& =& \lambda \displaystyle\int_{[0,1]^n} f\; {\rm d}\mu_{1} + (1 - \lambda) \displaystyle\int_{[0,1]^n} f\; {\rm d}\mu_{2} \\
&= &\displaystyle\int_{[0,1]^n} f\; {\rm d} (\lambda  \mu_{1}) +  \displaystyle\int_{[0,1]^n} f\; {\rm d} [(1 - \lambda) \mu_{2}] \\
& =& \displaystyle\int_{[0,1]^n} f\; {\rm d} [(\lambda \mu_{1}) + (1 - \lambda) \mu_{2}].
\end{array}
$$ 
Thus, $\delta(s) = \delta(\lambda s_{1} + (1 - \lambda) s_{2}) = \lambda \delta(s_{1}) + (1 - \lambda) \delta(s_{2})$, which proves that $\delta$ is affine.
\end{proof}

Before showing the main result of this section (Theorem \ref{homo-ext} below), let us point out an immediate but interesting consequence of Proposition \ref{prop:affine} above which reveals a remarkable analogy between states of MV-algebras and states of product algebras. Indeed, the Kroupa-Panti theorem shows that for every positive integer $n$, the state space $\mathcal{S}_{MV}(n)$ of the  free MV-algebra over $n$-free generators is affinely isomorphic to $\mathcal{M}(n)$. Thus, in particular, $\mathcal{S}(n)$ and $\mathcal{S}_{MV}(n)$ are affinely isomorphic via an isomorphism which is defined in the obvious way. 

The main result of this section hence reads as follows.
\begin{theorem}\label{homo-ext}
The following are equivalent for a state $s: \FP(n) \to [0,1]$
\begin{enumerate}
\item $s$ is extremal;
\item $\delta(s)$ is a Dirac measure;
\item $s\in \mathcal{H}(\FP(n))$. 
\end{enumerate}
\end{theorem}
\begin{proof}
$(1) \Rightarrow (2)$. 
If $s$ is extremal then its corresponding measure $\delta(s)$ is extremal in the space of Borel probability measures on $[0,1]$ since by Proposition \ref{prop:affine} $\delta$ is affine,  whence it preserves extremality. 
Extremal Borel measures on $[0,1]$ are exactly Dirac measures (see for instance \cite[Corollary 10.6]{Mu12}), thus $\delta(s) = \delta_{x}$ for some $x \in [0,1]^{n}$.
\vspace{.2cm}

$(2) \Rightarrow (1)$. If $s$ is such that $\delta(s)$ is a Dirac measure, then it is extremal. Indeed, by way of contradiction, let us suppose that $s$ can be expressed as a convex combination of two states $s_{1}, s_{2}$, that is, $s = \lambda s_{1} + (1- \lambda) s_{2},\; \lambda \in (0,1),$ but this would mean that $\delta(s) = \delta(\lambda s_{1} + (1 - \lambda) s_{2}) = \lambda \delta(s_{1}) + (1 - \lambda) \delta(s_{2})$, which contradicts the extremality of $\delta(s)$.

Hence, we proved that $(1) \Leftrightarrow (2)$.
\vspace{.2cm}

$(3) \Rightarrow (2)$. Follows from Proposition \ref{homoprod}.
\vspace{.2cm}

$(2) \Rightarrow (3)$. Let us suppose that $\delta(s)$ is a Dirac measure $\delta(s) = \delta_{x}$, and let us prove that $s$ is a homomorphism. By Theorem \ref{int-repr}, for every $f \in \FP(n)$, 
$$
s (f) = \int_{[0,1]^n} f\; {\rm d} \,\delta_{x} =  f(x),
$$ 
thus clearly $s$ is a homomorphism to $[0,1]$.

Hence we proved $(2) \Leftrightarrow (3)$, which settles the proof.
\end{proof}


Thus, via Krein-Milman theorem, we obtain the following:

\begin{corollary}\label{cor:CC}
The state space $\mathcal{S}(n)$ is the convex closure of the set of product homomorphisms from $\mathcal{F}_{\mathbb{P}}(n)$ into $[0,1]$.
\end{corollary}

\begin{remark}
For every $n$, the set of extremal states of the free MV-algebra $\mathcal{F}_{\mathbb{MV}}(n)$, with the topology inherited by restriction from the product space $[0,1]^{\mathcal{F}_{\mathbb{MV}}(n)}$, constitutes a compact Hausdorff space ${\rm ext}(\mathcal{S}_{MV}(n))$ which is homeomorphic to $[0,1]^n$ (see \cite[Theorem 2.5]{Mu} and \cite[Corollary 10.6]{Mu12}). Thus, ${\rm ext}(\mathcal{S}_{MV}(n))$ is closed.
A similar result for extremal product states is false. Indeed, as a consequence of \cite[Theorem 4.6]{KuMu},
$\mathcal{H}(n)={\rm ext}(\mathcal{S}(n))$ is not closed in the Tychonoff cube $[0,1]^{\FP(n)}$. Thus, it cannot be homeomorphic to $[0,1]^n$. 
However, Theorem \ref{homo-ext} still provides us with a bijection between ${\rm ext}(\mathcal{S}_{MV}(n))$, ${\rm ext}(\mathcal{S}(n))$ and $[0, 1]^n$. 
\end{remark}

\section{A logic to reason about the probability of product logic events}\label{sec:logic}

In this section we define a logic to reason about probabilities (in the sense of states) of product logic events. The idea is to follow the same fuzzy logic approach that has been used in the literature to formalise reasoning with different models of uncertainty, like probabilistic \cite{HGE,H98,FG07}, possibilistic \cite{FGM11} or evidential models \cite{FGM13}. 

The logic we will define, $FP(\Pi, \L_\Delta)$, is a two-tiered logic: an inner logic to represent the events (which will be product logic), and an outer logic to reason about the probability of the inner logic events. To express the additivity property of states, the outer logic will be a 
suitable (modal-like) extension of {\L}ukasiewicz logic:  
for each product logic proposition $\varphi$, $P\varphi$ will be an atomic modal formula in the outer logic that will be read as ``$\varphi$ is probable''. Note that we will not allow the nesting of the modality $P$. 

In more detail, the language of $FP(\Pi, \L_\Delta)$ contains the following sets of formulas: 

\begin{itemize}

\item Non-modal formulas: built from a countable set of propositional variables using product logic connectives, i.e.\ propositional product logic formulas.

\item Atomic modal formulas: of the form $P \varphi$, where $\varphi$ is a non-modal formula (of product logic).
 
\item Modal formulas: built from atomic modal formulas using $\L_\Delta$ logic\footnote{$\L_\Delta$ is the expansion of {\L}ukasiewicz logic   
with the Baaz-Monteiro projection connective $\Delta$, see \cite[\S 2.4]{H98}.} connectives. 

\end{itemize}
We will denote by $Fm$ the set of non-modal formulas and by $PFm$ the set of all modal formulas of $FP(\Pi, \L{_\Delta})$. In the following, by a {\em modal theory} we will refer to an arbitrary set of modal formulas. 

We will provide semantics for $FP(\Pi, \L{_\Delta})$ based on states on product logic formulas, that is, by mappings $\sigma: Fm \to [0, 1]$ satisfying the following conditions: 

\begin{itemize}
\item[S1.] $\sigma(\top)=1$ and $\sigma(\bot)=0$,
\item[S2.] $\sigma(\varphi \wedge \psi)+\sigma(\varphi \vee \psi)=\sigma(\varphi)+\sigma(\psi)$,
\item[S3.] If $\vdash_\Pi \varphi\to \psi$, then $\sigma(\varphi)\leq \sigma(\psi)$,
\item[S4.] If $\not \vdash_\Pi \neg\varphi$, then $\sigma(\varphi) = 0$ implies $\sigma(\neg\neg \varphi) = 0$.
\end{itemize}
Note that, due to S3, logical equivalence is preserved by states on formulas, that is, if $\varphi$ and $\psi$ are logically equivalent product logic formulas, then necessarily $\sigma(\varphi)  = \sigma(\psi)$ for any state $\sigma$. 

\begin{remark}  It is worth noticing that, if  $\sigma$ is a state on $Fm$, its  restriction $\sigma^n$ on $FM^n$, the set of formulas built from a finite subset of propositional variables $\{p_1, \ldots, p_n\}$, is again a state (in the sense that $\sigma^n: FM^n \to [0, 1]$ satisfies all the above properties $S1$ - $S4$).  Therefore, there is a one-one correspondence between states on formulas of $FM^n$ and states on the free algebra $\FP(n)$. 
\end{remark}

Following \cite[\S 7.2]{Montagna}, interpretations for  $FP(\Pi, \L{_\Delta})$-formulas will consist of pairs $(e, \sigma)$, where $e$ is a $[0, 1]$-evaluation of propositional variables, that extends to propositional formulas as usual with product logic truth-functions from $[0, 1]_\Pi$, and $\sigma$ is a state on $Fm$. Every interpretation  $(e, \sigma)$ assigns a truth-value $\| \Phi \|_{e, \sigma} \in [0,1]$ to every $FP(\Pi, \L{_\Delta})$-formula $\Phi$ as follows: 

\begin{itemize}
\item If $\Phi = \varphi$ is a propositional formula from $Fm$, $\| \Phi \|_{e, \sigma}  = e(\varphi)$.
\item if $\Phi = P \varphi$ is an atomic modal formula, $\| \Phi \|_{e, \sigma}  = \sigma(\varphi)$.
\item If  $\Phi$ is a propositional combination with $\L_\Delta$ connectives, then  $\| \Phi \|_{e, \sigma}$ is computed from its atomic modal subformulas by using their truth-functions from $[0, 1]_{\textrm{\L}_\Delta}$.
\end{itemize}



%
%
%
%
%
Note that for non-modal formulas $\varphi \in Fm$, $\| \varphi \|_{e, \sigma}$ only depends on $e$, while for modal formulas $\Phi \in PFm$, $\| \Phi \|_{e, \sigma}$ only depends on the state $\sigma$. Therefore, for the sake of a simpler notation, we will also write $\| \varphi \|_{e}$ and $\| \Phi \|_{\sigma}$ respectively. Now we define the following notion of logical consequence for $FP(\Pi, \L{_\Delta})$. 

\begin{definition}
Let  $\Gamma \cup \{\Phi\}$ be a (arbitrary) set of $FP(\Pi, \L{_\Delta})$-formulas.  Then define $\Gamma \models_{\FPL} \varphi$  if, for any interpretation $(e, \sigma)$, it holds that if $\| \Psi \|_{e, \sigma} = 1$ for all $\Psi \in \Gamma$, then $\| \Phi \|_{e, \sigma} = 1$. 
\end{definition}

%
%

As for the axiomatization of  $ \models_{\FPL}$, we need to properly capture properties $S1 - S4$ of states in terms of product logic formulas.  
Actually $S1, S2$, and $S3$ can be suitably encoded only using  the  language of {\L}ukasiewicz logic 
 with the following schemes: 

\begin{itemize}

%

\item[P1.] $P \top$, $\neg P \neg \bot$,


\item[P2.] $P(\varphi \lor \psi) \leftrightarrow P \varphi \oplus (P \psi  \ominus P(\varphi \land \psi))$,

\item[P3.] $P\varphi \to P \psi$, for $\varphi, \psi$ such that $\vdash_\Pi \varphi \to \psi$.



\end{itemize}
However, the axiom S4 of Definition \ref{def-State} cannot be written within the language of {\L}ukasiewicz logic, since this logic cannot express that a formula is not totally false. This is the reason for considering   {\L}$_\Delta$, the expansion of  {\L}ukasiewicz logic with the well-known Monteiro-Baaz $\Delta$ operator, for the outer logic. Indeed, using the language of  {\L}$_\Delta$,  
then S4 can be encoded by the following scheme: 

\begin{itemize}

\item[P4.] $\Delta(\neg P\varphi) \to \neg P\neg\neg \varphi$ , for $\varphi$ such that  $\not \vdash_\Pi \neg \varphi$ .


\end{itemize}
As outlined in \cite[\S 7.2]{Montagna} and in the proof of Theorem \ref{soundComplProb} below, the usual technique to prove completeness for a probabilistic modal logic as $FP(\Pi, \L{_\Delta})$ consists, mainly, in the following steps: (1) translating, at the propositional level of the outer logic (in this case $\L_\Delta$), all  modal axioms and rules; (2) using the completeness of the outer logic with respect to a standard algebra, build a model for the probabilistic modal logic. The typical problem of this strategy is that, as in this specific case, the propositional translation of the modal axioms leads to an infinite theory for $\L_\Delta$ which, however, is not strongly complete with respect to the algebra on $[0,1]$, i.e., if $\Gamma\cup\{\varphi\}$ is an infinite set of propositional formulas of $\L_\Delta$, it might happen that, although every model of $\Gamma$ is a model of $\varphi$, $\varphi$ cannot be proved from $\Gamma$.
%
%
Therefore, we will need to equip the outer logic $\L_\Delta$ with the following infinitary rule that makes it strongly complete (see \cite{Montagna} for full details): 
$$(IR) \; \frac{\Theta\vee(\Phi \to \Psi^n), \mbox{ for each } n \in \mathbb{N}}{\Theta\vee(\neg \Phi \lor \Psi)}$$

For this reason, we will henceforth extend the outer logic $\L_\Delta$ with the previous rule $(IR)$ and we will denote it by $\L_\Delta^+$. 

\begin{definition}  $FP(\Pi,\L_\Delta)$ is the logic, in the language defined above, whose axioms and rules are the following:  
\begin{itemize}

\item[($\Pi$)] Axioms and rule of product logic for non-modal formulas

\item[(\L$_\Delta$)] Axioms and rules of {\L}$_\Delta$ for modal formulas

\item[(P1)] $P \top, \neg P \neg \bot$

\item[(P2)] $P(\varphi \lor \psi) \leftrightarrow P \varphi \oplus (P \psi  \ominus P(\varphi \land \psi))$


\item[(P3)] $P\varphi \to P \psi$, \quad for every $\varphi, \psi$ such that $\vdash_\Pi \varphi \to \psi$

\item[(P4)] $\Delta(\neg P\varphi) \to \neg P\neg\neg \varphi$,  \quad for every $\varphi$ such that   $\not \vdash_\Pi \neg \varphi$

\item[(IR)]  from $\{\Theta\vee(\Phi \to \Psi^n) \mid  n  \in \mathbb{N}\}$ infer $\Theta\vee(\neg \Phi \lor \Psi)$,  \quad  for $\Theta,\Phi, \Psi \in PFm$

\end{itemize}
\end{definition}
Note that since we have two arrows $\to$ in $FP(\Pi, \L_\Delta)$, the inner one (from product logic $\Pi$) and the outer one (from the logic \L$_\Delta$), and two Modus Ponens
rules, one for each arrow. Moreover, in the outer logic we have the necessitation rule for $\Delta$ for modal formulas. 

Notice that in $FP(\Pi, \L_\Delta)$ the presence of the infinitary rule (IR) requires to slightly change the notion of {\em proof} in such a way ensuring that    
if,  for every $n\inÊ\mathbb{N}$,  we have a proof  of $\Theta\vee(\Phi\to \Psi^n)$ from the same set of premises $\Gamma$, then we also have a proof of $\Theta\vee(\neg\Phi\vee \Psi)$ from $\Gamma$ as well.\footnote{Formally, a proof of a formula $\Phi$ from $\Gamma$ is defined as a well-founded tree (i.e. with of possibly infinite width and depth, but with no branches of infinite length) where (i) the root $\Phi$ can have an infinite degree, (ii) the leaves are formulas from $\Gamma$ or instances of the axioms of $FP(\Pi, \L_\Delta)$, and (iii) for each node of the tree with a formula $\Psi$ there is an inference rule in $FP(\Pi, \L_\Delta)$ deriving $\Psi$ from its predecessors.}


In the following we restrict ourselves to prove completeness for deductions from modal theories.

\begin{theorem}[Soundness and Completeness] \label{soundComplProb}
Let  $\Gamma \cup \{\Phi\} \subset PFm$ be an (arbitrary) modal theory. Then, $\Gamma \vdash_{\FPL} \Phi$ iff $\Gamma \models_{\FPL} \Phi$, that is, iff for every state $\sigma$, if  $\| \Psi \|_{\sigma} = 1$ for all $\Psi \in \Gamma$, then $\| \Phi \|_{ \sigma} = 1$ as well. 
\end{theorem}

%
%
%
%

\begin{proof}  
Soundness is easy. As for completeness, we apply the usual technique in this kind of modal-like fuzzy probabilistic  logics of inductively defining a translation mapping $(\ )^*$ from the modal language $PFm$ into the propositional  \L$_\Delta^+$-language built from atomic modal formulas $P\varphi$ taken as propositional variables (see \cite{FG07,FGM11b} and \cite[\S 7.2]{Montagna}). 

Accordingly, each proof of $\Phi$ from $\Gamma$ in $\FPL$ can be translated into a proof of $\Phi^*$ in the logic  {\L}$_\Delta^+$  from the set of (propositional) formulas $\Gamma^* \cup AX^*$, where $AX^*$ is the set of all translated instances of axioms $(S1)$-$(S4)$. And viceversa, every proof of $\Phi^*$ from $\Gamma^* \cup AX^*$ in  {\L}$_\Delta^+$ gives rise to a proof of $\Phi$ from $\Gamma$ in $\FPL$. In other words, $\Gamma \vdash_{\FPL} \Phi$ iff $\Gamma^* \cup AX^* \vdash_{\textrm{\L}_\Delta^+} \Phi^*$. Note that, independently of whether $\Gamma$ is finite or not, $AX^*$ is infinite.  But now, since   {\L}$_\Delta^+$ is strongly complete \cite[Theorem 4]{Montagna}, $\Gamma^* \cup AX^* \vdash_{\textrm{\L}_\Delta^+} \Phi^*$ iff $\Gamma^* \cup AX^* \models_{\textrm{\L}_\Delta^+} \Phi^*$. 

Finally, we can check that  {\L}$_\Delta^+$-evaluations that are model of $AX^*$ are clearly in one-one correspondence with states on non-modal formulas. Namely, if $e$ is a  {\L}$_\Delta^+$-evaluation validating the translations of all instances of axioms (P1)--(P4), then the map $\sigma: Fm \to [0, 1]$ defined as $\sigma(\varphi) = e((P\varphi)^*)$ is state on product logic formulas, in the sense as defined above.  
In other words, $\Gamma^* \cup AX^* \models_{\textrm{\L}_\Delta^+} \Phi^*$ iff $\Gamma \models_{\FPL} \Phi$. This completes the proof. 
\end{proof}

\section{Conclusions and future work}

In this paper we have defined and studied the notion of state for free product algebras, i.e., the Lindenbaum-Tarski algebras of product logic. We may recall that Product logic is the third formalism which, together with \luk\ and G\"odel logics, stands at the ground of all continuous t-norm based logics, since any continuous t-norm can be obtained as an ordinal sum  of isomorphic copies of G\"odel t-norm (i.e. the minimum t-norm),   \luk\ t-norm and product t-norm. 

Our main result is a Kroupa-Panti-like representation for states. In other words we have proved that  our axiomatization of states captures the Lebesgue integral of product functions with respect to regular Borel probability measures and the relation between states and  measures is one-one. That result, besides supporting the appropriateness of our axiomatization, has several interesting consequences as welcome side effects. First of all, when studying the geometric properties of the state space,  it allows us to fully characterize extremal states in terms of $[0,1]$-valued product homomorphisms which, in turn, correspond one-one to Dirac measures on the space of extremal states. Furthermore the integral representation theorem shows that, for every natural number $n$, the state space of the free $n$-generated product algebra is affinely isomorphic to the state space of the free $n$-generated MV-algebra.  

In the last section of this paper, in order to point out the close relation between states and probabilistic logic, we  introduced a 
modal-like fuzzy logic obtained by the combination of product logic and a suitable expansion of \luk\ calculus  to reason about the probability of product logic events. The resulting logic turned out to be sound and complete with respect to the intended semantics given by states of free product algebras. 

The paper leaves several interesting open problems for further research. A first future direction clearly concerns with the generalization to the frame of product logic of a coherence (no-Dutch-book) criterion \emph{\`a la} de Finetti. In this regard,  the non-finiteness of free product algebras and the discontinuity of product implication, makes the problem of generalizing de Finetti's theorem to this setting non-trivial and hence particularly challenging. However, it is worth pointing out that the results contained in Section \ref{sec:extremal} pave the way for a first step in this direction.

Secondly, as \luk, G\"odel and product logics are the building blocks of H\'ajek logic BL, it is reasonable to think that the integral representation theorem for states of free product algebras, together with its analogous results for MV and G\"odel algebras, and the remarkable functional representation theorem for free BL-algebras \cite{AB}, are the necessary ingredients to shed a light on the problem of providing an appropriate axiomatization for states of free BL-algebras. 

\subsection*{Acknowledgments} The authors acknowledge partial support by the SYSMICS project (EU H2020-MSCA-RISE-2015 Project 689176). Also, Flaminio and Godo acknowledge partial support by the FEDER/MINECO project TIN2015-71799-C2-1-P.

 \appendix
\section{Proofs}\label{sec:proofs}

\subsection*{\bf Proof of Proposition \ref{lemma:tauQ}}

\begin{proof}
Let us start showing that $\tau_{\epsilon}^q$ is monotone.

Let $g, g' \in \mathscr{L}^q_\epsilon(n)$, with $g \leq g'$. In order to prove the monotonicity of $\tau_{\epsilon}^q$, we will show that for each $h \in \mathscr{L}_\epsilon(n)$ such that $h_{\rest G^q_\epsilon} = g'$ we can find  $k \in \mathscr{L}_\epsilon(n)$ such that 
\begin{equation}\label{eqProof1}
k_{\rest G^q_\epsilon} = g\mbox{ and }k \leq h. 
\end{equation}
Thus the claim will follow from the definition of $\tau_{\epsilon}^q$ and the monotonicity of $ \tau_{\epsilon}$. Let hence $h_{\rest G^q_\epsilon} = g'$ and let $l\in \mathscr{L}_\epsilon(n)$ which extends $g$. Thus, let $k=h\wedge l$. Clearly $k\in \mathscr{L}_\epsilon(n)$ and (\ref{eqProof1}) holds.

Now, we prove the linearity of $\tau_{\epsilon}^q$. 
First, we prove that if $h \in \mathscr{L}_\epsilon(n)$ is such that $h_{\rest G^q_\epsilon} = g + g'$, for some $g, g' \in \mathscr{L}^q_\epsilon(n)$, then there are $z, z' \in \mathscr{L}_\epsilon(n)$ such that $z_{\rest G^q_\epsilon} = g, z'_{\rest G^q_\epsilon} = g'$, and $h = z + z'$. Since $g\leq h_{\rest G^q_\epsilon}$, from the previous point, we know there is a $z\leq h$ and extending $g$. Thus, let $z'=h-z$.
Hence, $\tau_{\epsilon}^{q} (g + g')= \inf\{ \tau_{\epsilon}(h) \mid h_{\rest G^q_\epsilon} = g + g'\} = \inf\{ \tau_{\epsilon}(z + z') \mid z_{\rest G^q_\epsilon} = g, z'_{\rest G^q_\epsilon} = g'\}=\inf\{ \tau_{\epsilon}(z) + \tau_{\epsilon}(z') \mid z_{\rest G^q_\epsilon} = g, z'_{\rest G^q_\epsilon} = g'\}$, where the last equality follows by the linearity of $\tau_\epsilon$. 

Thus, we shall now prove the following: 
$\inf\{ \tau_{\epsilon}(z) + \tau_{\epsilon}(z') \mid z_{\rest G^q_\epsilon} = g, z'_{\rest G^q_\epsilon} = g'\} = \inf\{ \tau_{\epsilon}(z) \mid z_{\rest G^q_\epsilon} = g\} + \inf\{ \tau_{\epsilon}(z') \mid z'_{\rest G^q_\epsilon} = g'\}.$
 Since one inequality is obviously valid, we are left to prove that $$\inf\{ \tau_{\epsilon}(z) + \tau_{\epsilon}(z') \mid z_{\rest G^q_\epsilon} = g, z'_{\rest G^q_\epsilon} = g'\} \leq \inf\{ \tau_{\epsilon}(z) \mid z_{\rest G^q_\epsilon} = g\} + \inf\{ \tau_{\epsilon}(z') \mid z'_{\rest G^q_\epsilon} = g'\}.$$
 As to prove this claim, it suffices to notice that for any $z$ such that $z_{\rest G^q_\epsilon} = g$ and $z'$ such that $z'_{\rest G^q_\epsilon} = g'$ it is always possible to find a $\hat z  = \hat z + \hat z'$ where $\hat z \leq z, \hat z' \leq z'$, with $\hat z_{\rest G^q_\epsilon} = g$, $\hat z'_{\rest G^q_\epsilon} = g'$. Thus, being $\tau_{\epsilon}$ monotone, $\tau_{\epsilon}(\hat z) + \tau_{\epsilon}(\hat z') \leq \tau_{\epsilon}(z) + \tau_{\epsilon}(z')$, and the  claim is  settled.

 In a very similar way, we can show that $\tau_{\epsilon}^{q} (\lambda z) = \lambda \tau_{\epsilon}^{q}(z)$. Thus,  $\tau_{\epsilon}^{q}$ is linear.

Finally, in order to conclude the proof,  let $q_2 \leq q_1$ and let $g'\in \mathscr{L}^{q_2}_\epsilon(n)$ extending $g\in \mathscr{L}^{q_1}_\epsilon(n)$. Then, $\tau_{\epsilon}^{q_{1}} (g) \leq \tau_{\epsilon}^{q_{2}}(g')$  from the very definition of $\tau_\epsilon^q$.
\end{proof}
$$
\ast
$$

\subsection*{\bf Proof of Lemma \ref{lemmaSigmaEpsilon}}
\begin{proof}
The fact that $\sigma^q_{\epsilon}$ is a positive functional on $\mathscr{C}(G_\epsilon^q)$ follows by the very definition. In order to prove that $\sigma^q_{\epsilon}$ is monotone, let $c, c'\in \mathscr{C}(G^q_\epsilon)$ and assume $c\leq c'$. The following holds:
\begin{fact}\label{factX4}
For each sequence $\{g_1, g_2,\ldots\} \nnearrow\! c$, there is a sequence $\{g_1',g_2',\ldots\} \nnearrow\! c' $ and and index  $i_0$ such that, for every $i\geq i_0$, $g'_i\geq g_i$. 
\end{fact}
\begin{proof} (of Fact \ref{factX4}).
Let $\{g_1, g_2,\ldots\} \nnearrow\! c $, and $\{g^*_1, g^*_2,\ldots\} \nnearrow\! c'$. Define, for every $i$, $g_i'=g_i\vee g^*_i$ and easily check that this settles the claim. 
\end{proof}

Thus, we prove that $\sigma^q_{\epsilon}(c)\leq \sigma^q_{\epsilon}(c')$. Indeed, 
$$
\sigma^q_{\epsilon}(c)=\bigvee_{\overline{g}\in Seq(c)}\sigma_{\overline{g}}(c) =\bigvee_{\overline{g}\in Seq(c)}(\bigvee_{i\in \mathbb{N}} \tau^q_\epsilon(g_{i}))
$$
Fact \ref{factX4} ensures that, given a $\{g_{i}\}{\nnearrow\! c}$, there is a $\{r_{i}\}{\nnearrow\! c'}$ and, for every $i\geq i_0$, $g_{i}\leq r_{i}$, and hence, since $\tau^q_\epsilon$ is monotone, $\tau^q_\epsilon(g_{i})\leq \tau^q_\epsilon(r_{i})$.  Whence, for every  $\overline{g}\in Seq(c)$ there is  $\overline{r}\in Seq(c')$ such that $\sigma_{\overline{g}}(c) \leq \sigma_{\overline{r}}(c')$. Therefore 
$$
\sigma^q_{\epsilon}(c)=\bigvee_{\overline{g}\in Seq(c)} \sigma_{\overline{g}} (c) \leq \bigvee_{\overline{r}\in Seq(c')} \sigma_{\overline{r}}(c')=\sigma^q_{\epsilon}(c')
$$
showing that $\sigma^q_{\epsilon}$ is monotone.

Now, it is left to show that $\sigma^q_{\epsilon}$ is linear. To this end let us begin with the following claims:
\begin{fact}\label{factX5}
For every $c, c'\in \mathscr{C}(G^q_\epsilon)$ and for every $\lambda\in\mathbb{R}$, the following hold
\vspace{.2cm}

\noindent(1) For each $\{t_1,t_2,\ldots\}{\nnearrow\! c+c'}$, there are $\{a_1,a_2,\ldots\}{\nnearrow\! c}$ and $\{a_1', a_2', \ldots\}{\nnearrow\! c'}$ such that, for every $i$, $t_i=a_i+a_i'$.
\vspace{.2cm}

\noindent(2) For each $\{t_1,t_2,\ldots\}{\nnearrow\! \lambda c}$, there is $\{a_1,a_2,\ldots\}{\nnearrow\! c}$  such that, for every $i$, $t_i=\lambda a_i$.
\end{fact}
\begin{proof}(of Fact \ref{factX5}). 
(1) Let $\{t_1,t_2,\ldots\}$ be as given by hypothesis  and let $\{a_1,a_2,\ldots\}$ be any sequence converging to $c$. Then let, for every $i$,  $a_i'=t_i-a_i\in \mathscr{L}^q_\epsilon(n)$. Thus, $\{a'_1,a'_2,\ldots\}{\nnearrow\! (c+c')-c}$, that is $\{a'_1,a'_2,\ldots\}{\nnearrow\! c'}$ and this settles the claim.
\vspace{.2cm}

\noindent(2) Let $\{t_1,t_2,\ldots\}$ as in the hypothesis  and since $\lambda\neq 0$ put, for every $i$, $a_i=t_i/\lambda$. Thus, $\{a_1,a_2,\ldots\}{\nnearrow\! \lambda c/\lambda}$, that is,  $\{a_1,a_2,\ldots\}{\nnearrow\! c}$.
\end{proof}   

Now we prove that $\sigma^q_\epsilon$ is linear. Let $c, c'\in \mathscr{C}(G^q_\epsilon)$. Then 
$$
\sigma^q_\epsilon(c+c')= \bigvee_{\overline{t}\in Seq(c+c')} \sigma_{\overline{t}} (c+c').
$$
Fact \ref{factX5}(1) shows that, for each $\{t_1,t_2,\ldots\}{\nnearrow\! (c+c')}$, we can find $\{a_1,a_2,\ldots\}{\nnearrow\! c}$ and $\{b_1, b_2, \ldots\}{\nnearrow\! c'}$ such that, for every $i$, $t_i=a_i+b_i$. Thus, $\sigma_{\overline{t}}(c+c')=\bigvee_{i\in\mathbb{N}}\tau^q_\epsilon(a_i+b_i)$ and since $\tau^q_\epsilon$ is linear, $\tau^q_\epsilon(a_i+b_i)=\tau^q_\epsilon(a_i)+\tau^q_\epsilon(b_i)$. Thus,
$$
\begin{array}{lll}
\sigma_{\overline{t}}(c+c')&=&\bigvee_{i\in \mathbb{N}} \tau^q_\epsilon(a_i)+\tau^q_\epsilon(b_i)\\
&=&\lim_{i\in \mathbb{N}}\tau^q_\epsilon(a_i)+\tau^q_\epsilon(b_i)\\
&=&\lim_{i\in \mathbb{N}}\tau^q_\epsilon(a_i)+ \lim_{i\in \mathbb{N}}\tau^q_\epsilon(b_i)\\
&=&\sigma_{\overline{a}}(c)+\sigma_{\overline{b}}(c'),
\end{array}
$$ 
where the previous limits exist because every sequence $\{a_i\}$ and $\{a_i'\}$ is bounded by $c$ and $c'$ which are continuous functions and $\{a_i\}$ and $\{a_i'\}$ converge to $c$ and $c'$ on the compact set  $G^q_\epsilon$. 
In a similar way, we can prove that 
$$
\begin{array}{lll}
\sigma^q_\epsilon(c+c')&=&  \bigvee  \{\sigma_{\overline{t}} (c+c') \mid  \overline{t}\in Seq(c+c') \} \\ 
&=&  \bigvee  \{\sigma_{\overline{t}} (c+c') \mid  \overline{t} = \overline{a} + \overline{b}, \overline{a}\in Seq(c), \overline{b} \in Seq(c') \}  \\ 
&=&  \bigvee  \{\sigma_{\overline{a}}(c) + \sigma_{\overline{b}}(c') \mid \overline{a}\in Seq(c), \overline{b} \in Seq(c') \}  \\ 
 &=& \bigvee_{\overline{a}\in Seq(c')} \sigma_{\overline{a}} (c) +  \bigvee_{\overline{b}\in Seq(c')} \sigma_{\overline{b}} (c') \\ 
&=&\sigma^q_\epsilon(c)+\sigma^q_\epsilon(c')
\end{array}
$$

Finally, using a similar argument, but using Fact \ref{factX5}(2), $\sigma_\epsilon^q(\lambda c)=\lambda \sigma_\epsilon^q(c)$ so proving that $\sigma_{\epsilon}^q$ is linear.

In order to conclude the proof, notice that, for each $g\in \mathscr{L}^q_\epsilon(n)$, the constant sequence $\{g\}$ belongs to $Seq(g)$, and for any other sequence $\overline{t} = \{t_1, t_2 \ldots\}{\nnearrow\! g}$ we have $t_i\leq g$, whence 
$$
\sigma^q_\epsilon(g)=\bigvee_{\overline{t} \in Seq(g)}(\bigvee_{i\in \mathbb{N}} \tau^q_\epsilon(t_i))=\bigvee_{i\in \mathbb{N}} \tau^q_\epsilon(g)=\tau^q_\epsilon(g).
$$  
\end{proof}
\begin{center}
$\ast$
\end{center}
\subsection*{\bf Proof of Lemma \ref{lemma:riesz}}

\begin{proof}
Let, for every $q\in \mathbb{Q}$,  
 $\mu_\epsilon^q$ be a Borel measure that provides an integral representation of $\tau_\epsilon^q$ (Corollary \ref{cor:int}). 
Let us define for each $\mu_{\epsilon}^{q}$, the map $\hat\mu_\epsilon^{q}$ over the Borel subset of $G_{\epsilon}$ in the following way: 
$$
\hat\mu_\epsilon^{q}(B) = \mu_{\epsilon}^{q}(B \cap G^q_\epsilon).
$$ 
From Proposition \ref{lemma:tauQ} the sequence $\{\hat\mu_\epsilon^{q}\}$ is increasing and clearly  bounded. Thus, by the Vitali-Hahn-Saks theorem \cite[\S III.10]{Doob}, it converges to a $\sigma$-additive measure $\mu_{\epsilon}$.  Further notice that, by Corollary \ref{cor:BorelMeasure}, for every Borel subset $X$ of $G_\epsilon^q$, 
\begin{equation}\label{eq:borelMu}
\mu_\epsilon(X)=\mu_\epsilon^q(X)=\hat\mu_\epsilon^{q}(X).
\end{equation}

Now, let us define for each $g\in \mathscr{L}_\epsilon(n)$, the function $g_{q}: G_{\epsilon} \to [0,1]$ which equals $g_{\rest G^q_\epsilon}$ over $G^q_\epsilon$ and takes $0$ outside. Observe that each $g_{q}$ is not continuous but it is measurable. Clearly, each sequence $\{g_{q}\}_{q \in \mathbb{Q}}$ is non-decreasing and it converges pointwise to $g$: $\lim_{q} g_{q}(x) = g(x)$, for every $x \in G_{\epsilon}$. Then, by Levi's theorem (cf. \cite[\S 30, Theorem 2]{KoF}), 
$$
\lim_{q} \int_{G_{\epsilon}} g_{q} {\rm d}\mu_{\epsilon} = \int_{G_{\epsilon}} g {\rm d}\mu_{\epsilon}.
$$
Finally, observe that 
$$
\tau_{\epsilon}^{q}(g_{\rest G^q_\epsilon}) =  \int_{G^q_\epsilon} g_{\rest G^q_\epsilon} {\rm d}\mu_{\epsilon}^{q} =  \int_{G_{\epsilon}} g_{q} {\rm d}\hat\mu_\epsilon^{q},
$$ 
and also, by the definition of $\tau_{\epsilon}^{q}$ and Proposition \ref{lemma:tauQ}, $\lim_{q}\tau_{\epsilon}^{q}(g_{\rest G^q_\epsilon}) = \tau_{\epsilon}(g)$.
Thus,  $$\tau_{\epsilon}(h) = \lim_{q}\tau_{\epsilon}^{q}(g_{\rest G^q_\epsilon}) = \lim_{q}  \int_{G_{\epsilon}} g_{q} {\rm d}\hat\mu_\epsilon^{q} = \lim_{q} \int_{G_{\epsilon}} g_{q} {\rm d}\mu_{\epsilon} = \int_{G_{\epsilon}} g {\rm d}\mu_{\epsilon},$$
where the third equality follows from (\ref{eq:borelMu}) recalling that $g_q(y)=0$ for each $y\in G_\epsilon\setminus G_\epsilon^q$.
\end{proof}
$$
\ast
$$

\subsection*{\bf Proof of Proposition \ref{prop:closed}}
\begin{proof}
Let $\{s_i\}_{i\geq 0}$ be a sequence of states of $\FP(n)$ such that $\lim_{i \in \mathbb{N}} s_i = s$ exists, and let us prove that such $s$ is a state. Condition $S1$ of Definition \ref{def-State} is clearly verified. Let us show that $s$ respects condition $S2$. We need to prove that $s(f\lor g) = s(f) + s(g) - s(f \land g)$. Being each $s_n$ a state, we have that: $$\displaystyle\lim_{i \in \mathbb{N}} s_n (f \lor g) = \displaystyle\lim_{i \in \mathbb{N}} (s_n(f) + s_n(g) - s_n(f \land g))$$ and also, it clearly holds that: $$\displaystyle\lim_{i \in \mathbb{N}} (s_n(f) + s_n(g) - s_n(f \land g)) = \displaystyle\lim_{i \in \mathbb{N}} s_n (f) + \displaystyle\lim_{i \in \mathbb{N}} s_n(G) - \displaystyle\lim_{i \in \mathbb{N}} s_n(f \land g),$$ thus the claim directly follows.
It is easy to prove condition $S3$, since given $f, g \in \FP(n)$, if $f \leq g$ then $s_n(f) \leq s_n(g)$ for every $n \in \mathbb{N}$. Thus, it follows that: $$s(f) = \displaystyle\lim_{i \in \mathbb{N}}s_n(f) \leq \displaystyle\lim_{i \in \mathbb{N}}s_n(g) = s(g).$$
Let us finally prove $S4$. Let $f \in \FP(n), f \neq 0$, such that $s(f) = 0$. We shall prove that $s(\neg\neg f) = 0$.  Let $\supp(f)=\{x\in [0,1]^n\mid f(x)>0\}$. Then $\supp(f)$ is a union of $G_\epsilon$'s, whence it is a Borel subset of $[0,1]^n$. This observation, with Corollary \ref{cor:main}, imply that:
$$
s(f)=\lim_i\int_{[0,1]^n}f\;{\rm d}\mu_i=\lim_i\int_{\supp(f)}f\;{\rm d}\mu_i=0.
$$

\begin{fact}\label{claimm}
If $\lim_i\int_{\supp(f)}f\;{\rm d}\mu_i=0$ then $\lim_i\mu_i(\supp(f))=0$.
\end{fact}

\begin{proof} (of Fact \ref{claimm}) 
As we already noticed, $\supp (f) = \bigcup_{\epsilon \in \Sigma^{*}} G_{\epsilon}$, for some $\Sigma^{*}\subseteq\Sigma$ .  Thus, if $\lim_{i \in \mathbb{N}} \int_{\supp(f)}f\;{\rm d}\mu_i=0 $, and since the $G_\epsilon$'s are disjoint, the following holds:
$$
\lim_{i \in \mathbb{N}} \int_{\supp(f)}f\;{\rm d}\mu_i=\lim_{i \in \mathbb{N}} \int_{\bigcup_{\epsilon\in \Sigma^*}G_\epsilon}f\;{\rm d}\mu_i =\lim_{i \in \mathbb{N}}  \left(\sum_{\epsilon \in \Sigma^{*}} \int_{G_{\epsilon}} f {\rm d}\mu_i \right)  = \sum_{\epsilon \in \Sigma^{*}} \left( \lim_{i \in \mathbb{N}} \int_{G_\epsilon} f {\rm d}\mu_i\right).
$$ 
Therefore, $\sum_{\epsilon \in \Sigma^{*}} ( \lim_{i \in \mathbb{N}} \int_{G_\epsilon} f {\rm d}\mu_i)=0$, whence
 $\lim_{i \in \mathbb{N}} \int_{G_\epsilon} f {\rm d}\mu_i = 0$ for all $\epsilon \in \Sigma^{*}$.
Now, $G_{\epsilon}= \bigcup_{q \in \mathbb{Q} \cap (0,1]} G^q_{\epsilon}$, and hence 
$$
\int_{G_{\epsilon}} f {\rm d}\mu_i = \sup_{q \in \mathbb{Q} \cap (0,1]} \left(\int_{G^q_\epsilon} f {\rm d}\mu_i \right).
$$ 
Therefore: $$0 = \lim_i \int_{G_e} f {\rm d}\mu_i = \lim_i \sup_{q \in \mathbb{Q} \cap (0,1]} \left(\int_{G^q_e} f {\rm d}\mu_i \right) \geq \sup_{q \in \mathbb{Q} \cap (0,1]}  \lim_i  \left(\int_{G^q_e} f {\rm d}\mu_i\right).$$
Hence, for all $q \in \mathbb{Q} \cap (0,1]$, it follows that $\lim_i  (\int_{G^q_\epsilon} f {\rm d}\mu_i) = 0$. 
But since $G^q_\epsilon$ is compact, and $f$ is strictly positive on it, $\lim_i \mu_i(G^q_e) = 0$. Indeed, let $r=\min\{f(x)\mid x\in G_\epsilon^q\}$, thus $\lim_{i \in \mathbb{N}}  (\int_{G^q_\epsilon} r\, {\rm d}\mu_i) \leq \lim_{i \in \mathbb{N}}  (\int_{G^q_\epsilon} f {\rm d}\mu_i) = 0$. Thus, 
$$
0 = \lim_{i \in \mathbb{N}}  \left(\int_{G^q_\epsilon} r\, {\rm d}\mu_i\right) = r \cdot \lim_{i \in \mathbb{N}} \mu_i(G^q_e),
$$ 
that is, $\lim_{i \in \mathbb{N}}\mu_i(G^q_e) = 0$.
Therefore,
$$ \lim_{i \in \mathbb{N}} \mu_i\left(G_\epsilon\right) = \lim_{i \in \mathbb{N}} \mu_{i} \left(\bigcup_{q \in \mathbb{Q} \cap (0,1]} G^q_{\epsilon}\right) \leq \lim_{i  \in \mathbb{N}}\sum_{q \in \mathbb{Q} \cap (0,1]} \mu_{i}\left(G^q_{\epsilon}\right) =  \sum_{q \in \mathbb{Q} \cap (0,1]}  \lim_{i  \in \mathbb{N}} \mu_{i}\left(G^q_{\epsilon}\right) = 0$$
that is to say, 
 $\lim_{i \in \mathbb{N}} \mu_i(G_\epsilon) = 0$ for all $\epsilon \in \Sigma^{*}$.
%
%
%
%
%
%
%
%
Hence, finally,
$$\displaystyle\lim_{i  \in \mathbb{N}} \mu_i(\supp f) = \lim_{i  \in \mathbb{N}} \sum_{\epsilon \in \Sigma^{*}} \mu_i(G_\epsilon) =  \sum_{\epsilon \in \Sigma^{*}} \lim_{i  \in \mathbb{N}} \mu_i(G_e) = 0.$$
%
%
\end{proof}

Now, since $\supp(f) = \supp(\neg\neg f)$, and $(\neg\neg f) (x) = 1$ for every $x \in \supp(f)$,
%
$$
\begin{array}{lll}
s(\neg \neg f) &=& \displaystyle\lim_{i \in \mathbb{N}}s_i( \neg \neg f)\\  
&=&  \displaystyle\lim_{i \in \mathbb{N}}\displaystyle\int_{[0,1]^n} \neg \neg f\;{\rm d}\mu_i \\
&=&  \displaystyle\lim_{i \in \mathbb{N}}\displaystyle\int_{\supp(\neg\neg f)} \neg \neg f\;{\rm d}\mu_i \\
&=&  \displaystyle\lim_{i \in \mathbb{N}}\displaystyle\int_{\supp(f)} \neg \neg f\;{\rm d}\mu_i \\
&=&  \displaystyle\lim_{i \in \mathbb{N}}\displaystyle\int_{\supp(f)} 1\;{\rm d}\mu_i \\
&=&  \displaystyle\lim_{i \in \mathbb{N}} \mu_i(\supp(f))\\
& =& 0,
\end{array}
$$ 
where the last equality clearly follows from Fact \ref{claimm} above.
\end{proof}

\begin{thebibliography}{00}

\bibitem{ABGV17} S. Aguzzoli, M. Bianchi, B. Gerla, D. Valota. Probability Measures in G\"odel$_\Delta$ Logic. In Proc. of ECSQARU 2017, A. Antonucci et al. (eds.), Lecture Notes in Computer Science 10369,  Springer, 353-363, 2017. 

\bibitem{AB}
S. Aguzzoli, S. Bova. The free $n$-generated BL-algebra. {\em Annals of Pure and Applied Logic} 161(9): 1144--1170, 2010.

\bibitem{ABG}
S. Aguzzoli, S. Bova, B. Gerla, Free Algebras and Functional Representation for Fuzzy Logics. Chapter IX of {\em Handbook of Mathematical Fuzzy Logic--volume 2}, P. Cintula, P. H\'ajek, C. Noguera (Eds.), Studies in Logic, Mathematical Logic and Foundations, Studies in Logic, vol. 38, College Publications, London: 713--791, 2011.

\bibitem{AG10} S. Aguzzoli, B. Gerla. Probability Measures in the Logic of Nilpotent Minimum. {\em Studia Logica}, Volume 94, Issue 2, pp. 151--176, 2010. 

\bibitem{AGM1}
S. Aguzzoli, B. Gerla, V. Marra. De Finetti's no-Dutch-book criterion for G\"odel logic, {\em Studia Logica}, 90: 25--41, 2008.

\bibitem{AGM2} 
S. Aguzzoli, B. Gerla, V. Marra. Defuzzifying formulas in G\"odel logic through finitely additive measures. Proceedings of {\em The IEEE International Conference On Fuzzy Systems}, FUZZ-IEEE 2008, Hong Kong, China: 1886--1893, 2008.

\bibitem{Bir}
 G. Birkhoff. {\em Lattice theory}, Amer. Math. Soc. Colloquium Publications, rev. ed., 1948.

\bibitem{BS}
S. Burris, H.P. Sankappanavar. {\em A course in Universal Algebra}, Springer-Velag, New York, 1981.

\bibitem{BuKl} D. Butnariu, E. P. Klement. {\em Triangular Norm-Based Measures and Games with Fuzzy Coalitions}. Vol 10 of Theory and Decision Library Series, Springer, 1993.

\bibitem{BuKl2} D. Butnariu, E.P. Klement. Triangular norm-based measures. Chapter 23 of  {\em Handbook of Measure Theory}, E. Pap (Ed.), Elsevier Science, Amsterdam, 947--1010, 2002.

\bibitem{CG}
P. Cintula, B. Gerla. Semi-normal forms and functional representation of product fuzzy logic. {\em Fuzzy Sets and Systems}, 143(1):89--110, 2004.

\bibitem{Hand1}
P. Cintula, P. H\'ajek, C. Noguera (eds), {\em Handbook of Mathematical Fuzzy Logic} - volumes 1 and 2, Studies in Logic, Mathematical Logic and Foundations, vol. 37 and 38, 2011.

\bibitem{Doob}
J. L. Doob. {\em Measure Theory}, Graduate Texts in Mathematics, vol. 143. Springer-Verlag, New York, 1994.

\bibitem{FG07}
T. Flaminio, L. Godo. A logic for reasoning about the probability of fuzzy events. {\em Fuzzy Sets and Systems} 158(6): 625--638, 2007.

\bibitem{FGM11}
T. Flaminio, L. Godo, E. Marchioni. On the Logical Formalization of 
Possibilistic Counterparts of States over $n$-Valued {\L}ukasiewicz Events. {\em Journal of Logic and Computation}, 21(3): 429--446, 2011. 

\bibitem{FGM11b}
T. Flaminio, L. Godo, E. Marchioni. Reasoning about Uncertainty of Fuzzy Events: an Overview. In {\em Understanding Vagueness - Logical, Philosophical, and Linguistic Perspectives}, P. Cintula et al. (Eds.), College Publications, pp. 367--400, 2011.

\bibitem{FGM13}
T. Flaminio, L. Godo, E. Marchioni. Logics for belief functions on MV-algebras, {\em International Journal of Approximate Reasoning}, 54(4): 491--512, 2013.

\bibitem{FK15}
T. Flaminio and T. Kroupa. States of MV-algebras. Chapter XVII of {\em Handbook of Mathematical Fuzzy Logic - volume 3}, C. Ferm\"uller, P. Cintula and C. Noguera (Eds.), Studies in Logic, Mathematical Logic and Foundations, vol. 58, College Publications, London, 2015.

\bibitem{Goodearl}
K. R. Goodearl. {\em Partially Ordered Abelian Group with Interpolation}. AMS Math. Survey and Monographs, vol 20, 1986.

\bibitem{HGE}
P. H\'ajek, L. Godo, F. Esteva,  Probability and Fuzzy Logic. In {\em Proc. of Uncertainty in Artificial Intelligence} UAI'95, P. Besnard and S. Hanks (Eds.), Morgan Kaufmann, San Francisco, 237--244 (1995). 

\bibitem{H98}
P. H\'ajek. {\em Metamathematics of Fuzzy Logics}, Kluwer Academic Publishers, Dordrecht, 1998.

\bibitem{KMP}
E. P. Klement, R. Mesiar and E. Pap. {\em Triangular Norms}, Kluwer
Academic Publishers, Dordrecht, 2000.

\bibitem{KoF}
A. N. Kolmogorov and S. V. Fomin. {\em Introductory Real Analysis}, Dover Publications, 1970.

\bibitem{KroupaInt}
T. Kroupa. Every state on semisimple MV-algebra is integral. {\em Fuzzy Sets and Systems} 157 (20), 2771--2787, 2006.

\bibitem{KuMu}
J. K\"uhr, D. Mundici.  De Finetti theorem and Borel states in 
$[0,1]$-valued algebraic logic. {\em International Journal of Approximate Reasoning }46 (3)
605--616, 2007.

\bibitem{Montagna}
F.~Montagna. Notes on strong completeness in {{\L}}ukasiewicz, product and {BL} logics and in their first-order extensions.
In {\em Algebraic and Proof-theoretic Aspects of Non-classical Logics}, S. Aguzzoli et al. (Eds.), LNAI 4460, Springer, pages 247--274, 2006.

\bibitem{MU15}
F. Montagna, S. Ugolini. A categorical equivalence for product algebras. {\em Studia Logica}, 103(2): 345--373, 2015.

\bibitem{Mu}
D. Mundici. Averaging the truth-value in {\L}ukasiewicz logic. {\em Studia Logica}, 55(1), 
113--127, 1995. 

\bibitem{Mu12}
 D. Mundici. {\em Advanced \luk\ calculus and MV-algebras}, Trends in Logic 35, Springer, 2011.

 \bibitem{Navara} M. Navara. Probability theory of fuzzy events. In  {\em Proc. of EUSFLAT-LFA 2005}, E. Montseny and P. Sobrevilla (Eds.), pp. 325-329, 2005. 

\bibitem{Panti}
G. Panti. Invariant measures on free MV-algebras.  {\em Communications in Algebra},  36(8):2849--2861, 2009.

\bibitem{Priestley} H.A.  Priestley. Ordered topological spaces and the representation of distributive lattices. {\em Proc. London Math. Soc.}, 24(3) 507-530, 1972. 

\bibitem{Ru}
W. Rudin. {\em Real and Complex Analysis}, 3 ed., McGraw-Hill, 1986.

\bibitem{Schwartz93}
N. Schwartz. Piecewise polynomial functions. In {\em Ordered Algebraic Structures}, J. Martinez and C. Holland (Eds.), Kluwer, pp. 169--202, 1993. 

\bibitem{Willard}
S. Willard. {\em General Topology}. Dover Publications, 2004.

\end{thebibliography}


\end{document}